\documentclass[11pt]{siamltex}
\usepackage[T1]{fontenc}
\usepackage[latin9]{inputenc}
\usepackage{array}
\usepackage{xcolor}
\usepackage{float}
\usepackage{mathtools}
\usepackage{multirow}
\usepackage{amsmath}
\usepackage{amssymb}
\usepackage{graphicx}
\allowdisplaybreaks
\usepackage{xcolor}
\usepackage[colorlinks,
linkcolor=red,
anchorcolor=yellow,
citecolor=blue,
]{hyperref}

\providecommand*{\Dist}[2]{\operatorname{dist}({#1};{#2})}   % distance
\providecommand*{\Dist}[2]{\Dist{#1}{#2}}
\newcommand{\Bn}{{\boldsymbol{n}}}

\newcommand{\Bp}{{\boldsymbol{p}}}
\newcommand{\Bq}{{\boldsymbol{q}}}
\newcommand{\Br}{{\boldsymbol{r}}}

\newcommand{\BL}{{\boldsymbol{L}}}

\newcommand{\BV}{{\boldsymbol{V}}}
\newcommand{\BW}{{\boldsymbol{W}}}

\newcommand{\lambdabf}{\boldsymbol{\lambda}}
\newcommand{\mubf}{\boldsymbol{\mu}}

\newcommand{\rmI}{{\rm I}}

\newcommand{\Cl}{\mathcal{L}}

\newcommand{\be}{\begin{eqnarray}}
\newcommand{\ee}{\end{eqnarray}}

\newcommand{\ben}{\begin{eqnarray*}}
\newcommand{\een}{\end{eqnarray*}}

\newtheorem{example}{Example}[section]
\newtheorem{rem}{Remark}[section]
\newtheorem{lem}{Lemma}[section]
\newtheorem{thm}{Theorem}[section]

\begin{document}

\title{A Uniform Preconditioner for a Newton Algorithm for Total-Variation Minimization and Minimum-surface Problems}

\author{Xue-Cheng Tai\thanks{Department of Mathematics, Hong Kong Baptist University, Kowloon Tong Kowloon, Hong Kong. (xuechengtai@gmail)}.
\and
Ragnar Winther\thanks{Department of Mathematics, University of Oslo}.
\and
Xiaodi Zhang
	\thanks{Corresponding Author. Henan Academy of Big Data, Zhengzhou University, Zhengzhou 450052, China. School of Mathematics and Statistics, Zhengzhou University, Zhengzhou 450001, China.  (zhangxiaodi@lsec.cc.ac.cn)} 
		\and
Weiying Zheng\thanks{LSEC, Academy of Mathematics and Systems Science, Chinese Academy of Sciences, Beijing, 100190, China. School of Mathematical Science, University of Chinese Academy of Sciences,Sciences, Beijing 100049, China. The forth author was supported in part by the National Science Fund for Distinguished Young Scholars 11725106 and by China NSF grant 11831016. (zwy@lsec.cc.ac.cn)}}
\maketitle

\begin{abstract} Solution methods for the nonlinear partial differential equation of the Rudin-Osher-Fatemi (ROF) and minimum-surface models are fundamental for many modern applications. Many efficient algorithms have been proposed. First order methods are common. They are popular due to their simplicity and easy implementation. Some second order Newton-type iterative methods have been proposed like Chan-Golub-Mulet method.
In this paper, we propose a new Newton\textendash Krylov solver for primal-dual
finite element discretization of the ROF model. The method is so simple that we just need to use some diagonal preconditioners during the iterations.
Theoretically,
the proposed preconditioners are further proved to be robust
and optimal with respect to the mesh size, the penalization parameter,
the regularization parameter, and the iterative step, essentially it is a parameter independent preconditioner.
We first discretize
the primal-dual system by using mixed finite element methods, and
then linearize the discrete system by Newton\textquoteright s method.
Exploiting the well-posedness of the linearized problem on appropriate
Sobolev spaces equipped with proper norms, we propose block diagonal
preconditioners for the corresponding system solved with the minimum residual
method.  Numerical results
are presented to support the theoretical results.
\end{abstract}


\section{Introduction\label{sec:intro}}

Image restoration is a fundamental and challenging task in image processing. A surge of research has been done in variational and PDE-based
approaches. The ROF model due to Rudin, Osher and Fatemi \cite{Rudin1992}
is one of the most successfully and widely used mathematical models.
Given an image $f:\Omega\subset\mathbb{R}^{d}\mapsto\mathbb{R},\,d=1,2,3$,
the ROF model is trying to solve the following minimization problem:
\begin{equation}
\min_{v}\left\{ E(v)=\int_{\Omega}\alpha|\nabla v|{\rm d}x+\frac{1}{2}\int_{\Omega}(v-f)^{2}{\rm d}x\right\} ,\label{eq:rof}
\end{equation}
where $f$ is the observed image, $\alpha>0$
is the penalization parameter which controls the trade-off between
goodness-of-fit and visibility in its minimizer $u$. The
Euler-Lagrange equation for the minimization problem \eqref{eq:rof}
can be formally written as
\begin{equation}
-\alpha\nabla\cdot\left(\frac{\nabla u}{|\nabla u|}\right)+u-f=0	\quad \mbox{  in  } \Omega,
	\qquad \nabla u \cdot \Bn = 0 \quad \mbox{  on  } \partial \Omega.\label{eq:rofmodelu}
\end{equation}
Here and after, $\Bn$ denotes the unit outer normal vector of $\partial \Omega$. This equation indeed characterizes the first-order optimality condition
of \eqref{eq:rof}, which is also known as the curvature equation
\cite{Osher2003}. As a fundamental well studied model in the literature, the ROF model is important for many modern  applications include scientific computing, image processing and data sciences.

To deal with the singularity caused by the
TV-norm minimization in \eqref{eq:rof}, the following regularized minimization problem is often studied,
\begin{equation}
	\min_{v}\left\{ E_{\beta}(v)=\alpha\int_{\Omega}
	\sqrt{|\nabla v|^{2}+\beta}{\rm d}x+\frac{1}{2}\int_{\Omega}(v-f)^{2}{\rm d}x\right\} ,\label{eq:rofp}
\end{equation}
where the regularization parameter $\beta>0$ is typically small. Since the regularized energy functional $E_{\beta}(v)$ is strictly convex for $\beta >0$,  the minimizer to \eqref{eq:rofp} is unique. In \cite{Acar1994}, it has been shown the solution of the regularized problem \eqref{eq:rofp}
converges to the solution of \eqref{eq:rof} as $\beta\rightarrow0$.
The convergence has been established rigorously in \cite{Yao2008}.
It should be noted that this regularization technique is mostly used to compute the minimizer of the total variation energy and its
variants \cite{Casas1999,Chan2003}.

It is well-known that first integral in \eqref{eq:rofp} is the surface area of the graph of function $v$ when $\beta =1$. Thus, model \eqref{eq:rofp} is essentially the minimum surface problem when $\beta =1$. In this work, we will design a fast algorithm which has good convergence properties uniformly with respect to $\beta$ which means that our algorithm works for the regularized total variation minimization   model as well as for the minimum surface model.

The Euler-Lagrange equation corresponds to the  minimization problem \eqref{eq:rofp} reads
\begin{equation}
	-\alpha\nabla\cdot\left(\frac{\nabla u}{\sqrt{|\nabla u|^{2}+\beta}}\right)+u-f=0
	\quad \mbox{  in  } \Omega,
	\qquad
	\nabla u \cdot \Bn = 0 \quad \mbox{  on  } \partial \Omega.
	\label{eq:rofpmodelu}
\end{equation}
We want to emphasis once more that our proposed method works uniformly with respect to $\beta \in (0, 1]$.
% It can be seen that equation \eqref{eq:rofmodelu} is the numerical regularized approximation of equation \eqref{eq:rofpmodelu} and

% {\color{red} \Large [Need to add some descriptions on minimum-surface models.]}

Numerical solution to the minimization problem \eqref{eq:rofp} poses a challenging problem due to
the presence of a highly nonlinear and non-differentiable term. To
get around this difficulty, a large effort has been devoted to construct
effective schemes for the minimization problem \eqref{eq:rofp} and
the PDE problems \eqref{eq:rofmodelu} in the past two decades. For
instance, the artificial time marching scheme in \cite{Marquina2000,Rudin1992},
the lagged diffusivity fixed-point method in \cite{Dobson1997,Vogel1996},
Chan-Golub-Mulet method in \cite{Chan1999}, the Bregman iteration
in \cite{Goldstein2009}, the augmented Lagrangian technique in \cite{Wu2010}
and some others in \cite{Bartels2012,Bartels2018,Bartels2015,Chambolle2004,Lee2019,Xu2010}.
However, most of these numerical algorithms are gradient-decent type
and thus only first-order. Thus, the main motivation of this work
is to develop an effective second-order algorithm for the model \eqref{eq:rofpmodelu}.

It is well-known that a ``good'' algorithm for the nonlinear problem
should include not only fast iterative methods but also fast linear
solvers for the linear systems obtained after linearization. Solving
the linear systems is usually the most important, challenging and
time-consuming part in the overall simulation, which is due to the
large-scale and ill-conditions of the linear systems. For
the underlying nonlinear PDE problem \eqref{eq:rofpmodelu}, the
condition number of the linear system tends to infinity when the mesh
size is approaching zero. Moreover, the variability of the parameters,
such as the penalization parameter $\alpha$ and the regularization parameter $\beta$, can additionally
influence the scale of the condition number of the system. However,
little work has been done to develop robust and efficient solvers
for the resulting linear systems. Here, the term ``robust'' posses
two underlying merits: (i) The convergence rate is independent of
the mesh size. (ii) The method is robust with respect to parameters.
In order to improve the efficiency of the numerical simulations, we
will also focus on designing robust preconditioners for the linearized
problem.

Based on the discussion carried out in the previous paragraphs, the
purpose of this paper is to propose a preconditioned Newton method
for primal-dual finite element discretization of the ROF model. We
shall adopt a primal-dual formulation of the  model \eqref{eq:rofpmodelu} which contains
the primal variable $u$, the dual variable $\Bp$ and the multiplier $\lambdabf$. We first discretize the primal-dual system
by using mixed finite element methods, and then linearize the discrete
nonlinear system by Newton\textquoteright s method. Following the
operator preconditioning framework in \cite{Mardal2011}, we develop
block diagonal preconditioners for the linearized problem.
The derivation exploits its well-posedness on appropriate Sobolev
spaces equipped with proper norms. We also rigorously prove the condition
number of the preconditioned operator is uniformly bounded by a
constant independent of the mesh size, the penalization parameter $\alpha$ and the regularization parameter $\beta$. Thus, the proposed preconditioners are robust with respect
to the mesh size, the smoothing parameter, the regularization parameter,
and the iterative step. As a product, we develop a second-order Newton
scheme with robust and optimal preconditioners for the  model \eqref{eq:rofpmodelu}.
Finally, numerical experiments are supplied to test the accuracy
of our schemes and validate the uniform robustness of the proposed preconditioners.

We want to emphasis that the essential ideas presented in this work is rather general and can be used for other nonlinear problems. In \cite{Winther2009}, the harmonic map problem was considered. A uniform preconditioner for the Newton iteration was also constructed using operator preconditioning framework published later in \cite{Mardal2011}. All these show that the methodology presented here can be used for a large class of problems.

The remainder of the paper is structured as follows. In Section \ref{sec:model},
we introduce a primal-dual formulation of the ROF model and present
its Newton iterations. Section \ref{sec:fem} is devoted to introducing
the finite element discretization and giving Newton\textquoteright s
linearization for the discrete problem. In Section \ref{sec:wellposed},
we derive the well-posedness of the discrete problem on
chosen spaces equipped with subtle norms. We propose and analyze
the robust preconditioners in Section \ref{sec:precond}. We carry
out several numerical experiments in Section \ref{sec:num} to
confirm the efficiency of our proposed algorithms. The paper ends
with a concluding remark in Section \ref{sec:conlu}.

\section{Our proposed model and Newton algorithm\label{sec:model}}
First we introduce some Sobolev spaces and norms used in this paper. Throughout the paper, we shall denote vector-valued quantities by boldface notations. Let $L^{2}(\Omega)$ be the usual Hilbert space of square integrable functions which is equipped with the following inner product and norm,
$$
\langle u, v\rangle :=\int_{\Omega} u(x) v(x) \mathrm{d} x, \qquad \left\Vert u\right\Vert\coloneqq \langle u, u\rangle^{1 / 2}.
$$
Let $H^{1}(\Omega)$ be its subspace with square integrable gradients, and its standard norm.
% is written for $\left\Vert \cdot \right\Vert$.
We also use the space $L^{\infty}$ with its canonical norm, $\|v\|_{\infty}=\underset{x\in\Omega}{\operatorname{ess} \sup }|v(x)|$. For a vector $\mathbf{x} = (x_1, x_2, \cdots, x_d) \in\mathbb{R}^{d}$,
we shall use the notation
\[
|\mathbf{x}|_{\beta}=\sqrt{|\mathbf{x}|^{2}+\beta}=\sqrt{\sum_{i=1}^{d}x_{i}^{2}+\beta}.
\]
For notation simplicity,  we will also use $\langle \cdot, \cdot\rangle$  throughout this work to denote $L^2$-type of  inner product for vectors, functions include vector functions and duality pairing. From the context where this notation is used, it is clear which inner product or duality this notation is referring to.
As usual, $\nabla$ and $\nabla \cdot$ will be used as (distributional)  gradient and divergence operators.

%{\color{red} \Large [???]}

%{\color{blue}[We will show that our method is independent of this parameter $\beta$ and thus provides an efficient numerical method both for ROF and minimum-surface problems (what is the minimum-surface problem?)]}.

Inspired by \cite{Wu2010}, we introduce the auxiliary variable $\boldsymbol{p}=\nabla u$ and reformulate \eqref{eq:rofp} into an equivalent constrained minimization problem:
\be
\min_{\substack{\Bp, u \\  \Bp=\nabla u} }\left\{ \int_{\Omega}\alpha|\Bp|_{\beta}{\rm d}x+\frac{1}{2}\int_{\Omega}(u-f)^{2}{\rm d}x
\right\}. \label{eq:constr_pro}
\ee
The constraint condition can be enforced by use of a Lagrange multiplier $\mubf$, and we then seek stationary points to the Lagrangian functional
\[
\Cl(\Bq,v,\mathbf{\mubf})=\int_{\Omega}
\left(\alpha|\Bq|_{\beta}+\frac{1}{2}(v-f)^{2}-\mathbf{\mubf}\cdot(\Bq-\nabla v)\right){\rm d}x.
\]
Let $(\Bp, u, \lambdabf)$ be one saddle point for this problem. As in \cite{Lao2021,Daniel2012,Tai2009},  the first-order optimality condition for \eqref{eq:constr_pro} is:
\begin{equation}
\alpha\Bp/|\Bp|_{\beta}-\mathbf{\lambdabf}=0,\quad u-\nabla\cdot\mathbf{\lambdabf}=f,\quad-\Bp+\nabla u=0\quad \mbox{  in  } \Omega, \label{eq:ROFmodel}
\end{equation}
in conjunction with the following boundary
conditions
\begin{equation}
\mathbf{\lambdabf}\cdot\Bn=\nabla u\cdot\Bn=0\qquad\text{ on }\partial\Omega,\label{eq:ROFbc}
\end{equation}
%where $\Bn$ is the unit out normal vector on $\partial \Omega.$
Different from the original Euler-Lagrange equation \eqref{eq:rofpmodelu}
for $u$, this system contains three variables, i.e. functions $u$, $\Bp$
and $\lambdabf$. Following \cite{Chan1999}, the method is dubbed as the primal-dual method.

To solve the nonlinear system \eqref{eq:ROFmodel}-\eqref{eq:ROFbc},
we apply Newton's method as the linearization technique. For convenience,
we write this system in the compact form $F(\Bp,u,\lambdabf)=0$,
where $F$ is the non-linear map given by
\[
F:(\Bp,u,\lambdabf)\mapsto\left(\alpha\Bp/|\Bp|_{\beta}-\lambdabf,u-f-\nabla\cdot\lambdabf,-\Bp+\nabla u\right).
\]
Therefore, Newton's method to solve this problem is: Given $\left(\Bp^{n},u^{n},\lambdabf^{n}\right)$,
compute $\left(\Bp^{n+1},u^{n+1},\lambdabf^{n+1}\right)$ by
\begin{equation}
\left(\Bp^{n+1},u^{n+1},\lambdabf^{n+1}\right)=\left(\Bp^{n},u^{n},\lambdabf^{n}\right)+\left(\delta\Bp^n,\delta u^n,\delta\lambdabf^n\right),\label{eq:update}
\end{equation}
where the correction $\left(\delta\Bp^n,\delta u^n,\delta\lambdabf^n\right)$
is defined by
\begin{equation}
DF\left(\Bp^{n},u^{n},\lambdabf^{n}\right)\left(\delta\Bp^n,\delta u^n,\delta\lambdabf^n\right)=-F\left(\Bp^{n},u^{n},\lambdabf^{n}\right).\label{eq:Newton}
\end{equation}
 As usual, $DF\left(\Bp^{n},u^{n},\lambdabf^{n}\right)$ is the
Fr{\'e}chet derivative of the operator $F$ at $\left(\Bp^{n},u^{n},\lambdabf^{n}\right)$.
For a given vector field $\Br$, for simplicity, assume that
$\Br(x)\neq0$ for all $x\in\Omega,$ define the matrix valued
function $H(\Br)=H(\Br(x))$ by
\[
H(\Br)=\frac{1}{|\Br|_{\beta}}\left(\mathrm{I}-\frac{\Br\Br^{t}}{|\Br|_{\beta}^{2}}\right).
\]
Here $\Br^{t}$ denotes the transpose of $\Br$. It
is easy to see that the matrix $H(\Br)$ is symmetric and satisfies
\begin{equation}
\frac{\beta}{|\Br|_{\beta}^{3}}\leq\frac{H(\Br)\boldsymbol{\mathbf{\xi}}\cdot\boldsymbol{\xi}}{\boldsymbol{\mathbf{\xi}}\cdot\boldsymbol{\xi}}\leq\frac{1}{|\Br|_{\beta}},\quad\forall\boldsymbol{\xi}\in\mathbb{R}^{d}\backslash\left\{ \boldsymbol{0}\right\} .\label{eq:MatrixH}
\end{equation}
In addition, the matrix $H(\Br)$ is uniformly elliptic provided
that $\Br\in\boldsymbol{L}^{\infty}(\Omega)$.

As a consequence, for the $n$-th step of the Newton iteration the
residual system \eqref{eq:Newton} can be written as:
\begin{equation}
\begin{cases}
\alpha H(\Bp^{n})\delta\Bp^n-\delta\lambdabf^n & =\Br^n_{\Bp},\\
\delta u^n-\nabla\cdot\delta\lambdabf^n & =r^n_{u},\\
-\delta\Bp^n+\nabla\delta u^n & =\Br^n_{\lambdabf},
\end{cases}\label{eq:Residual}
\end{equation}
where the right-hand sides are given by
\begin{align*}
\Br^n_{\Bp} & \coloneqq-\alpha\Bp^{n}/|\Bp^{n}|_{\beta}+\mathbf{\lambdabf}^{n},\\
r^n_{u} & \coloneqq f-u^{n}+\nabla\cdot\mathbf{\lambdabf}^{n},\\
\Br^n_{\lambdabf} & \coloneqq\Bp^{n}-\nabla u^{n}.
\end{align*}
In subsequent sections, additional details on the iterations will
be provided.

\section{Finite element discretization\label{sec:fem}}

In this subsection, we introduce the finite element discretization
of the system \eqref{eq:ROFmodel}. Let $\mathcal{T}_{h}$ be a quasi-uniform
and shape-regular simplex mesh of $\Omega$ with mesh size $h$.
%$h=\underset{T\in\mathcal{T}_{h}}{\max}h_{T}.$
For any integer $k\geq0, T\in\mathcal{T}_{h}$, let $P_{k}(T)$ be the space of polynomials
of degree $k$, and define $\boldsymbol{P}_{k}(T)= (P_{k}(T))^{d}$.
We associate a triple of piecewise polynomial, finite-dimensional
spaces to approximate the solution $\left(\Bp,u,\lambdabf\right)$,
\begin{align*}
\BV_{h} & \coloneqq\left\{ \Bq\in\boldsymbol{L}^{2}(\Omega):\left.\Bq\right|_{T}\in\boldsymbol{P}_{0}(T),\quad\forall T\in\mathcal{T}_{h}\right\} ,\\
U_{h} & \coloneqq\left\{ v\in H^{1}(\Omega):\left.v\right|_{T}\in P_{1}(T),\quad\forall T\in\mathcal{T}_{h}\right\} ,\\
\BW_{h} & \coloneqq\left\{ \mubf\in\boldsymbol{L}^{2}(\Omega):\left.\mubf\right|_{T}\in\boldsymbol{P}_{0}(T),\quad\forall T\in\mathcal{T}_{h}\right\} .
\end{align*}

Based on the above finite element spaces, the finite element approximation
of the system \eqref{eq:ROFmodel} is formulated as follows: Find
$\left(\Bp_{h},u_{h},\lambdabf_{h}\right)\in\BV_{h}\times U_{h}\times\BW_{h}$
such that for any $(\Bq_{h},v_{h},\mubf_{h})\in\BV_{h}\times U_{h}\times\BW_{h}$,
\begin{equation}
\begin{cases}
\langle \alpha\Bp_{h}/|\Bp_{h}|_{\beta},\Bq_{h}\rangle -\langle \mathbf{\lambdabf}_{h},\Bq_{h}\rangle & =0,\\
\langle u_{h},v_{h}\rangle +\langle \mathbf{\lambdabf}_{h},\nabla v_{h}\rangle  & =\langle f,v_{h}\rangle ,\\
-\langle \Bp_{h},\mubf_{h}\rangle +\langle \nabla u_{h},\mubf_{h}\rangle  & =0.
\end{cases}\label{eq:weakhfull}
\end{equation}
Since we are interested in developing fast solvers for the discrete
problem, we do not elaborate on the well-posedness of \eqref{eq:weakhfull}
and simply assume that it has a unique solution.

%{\color{red} \Large [Maybe Prof. Ragnar Winther argue this. But I have no idea about this.]}

We shall apply the Newton algorithm \eqref{eq:Newton} to the above finite element system as well.
Let $\left(\Bp_{h}^{n},u_{h}^{n},\lambdabf_{h}^{n}\right)\in\BV_{h}\times U_{h}\times\BW_{h}$
be the approximate solutions of \eqref{eq:weakhfull} at the $n$-th Newton
iteration. From the linearization in \eqref{eq:Residual}, for each
step of the Newton iteration, the residual equation reads: Find $\left(\delta\Bp^n_{h},\delta u^n_{h},\delta\lambdabf^n_{h}\right)\in\BV_{h}\times U_{h}\times\BW_{h}$
such that for any $(\Bq_{h},v_{h},\mubf_{h})\in\BV_{h}\times U_{h}\times\BW_{h}$,
\begin{equation}
\begin{cases}
\langle \alpha H(\Bp_{h}^{n})\delta\Bp^n_{h},\Bq_{h}\rangle-\langle \delta\lambdabf^n_{h},\Bq_{h}\rangle  & =R^n_{\Bp}\left(\Bq_{h}\right),\\
\langle \delta u^n_{h},v_{h}\rangle +\langle \delta\lambdabf^n_{h},\nabla v_{h}\rangle  & =R^n_{u}\left(v_{h}\right),\\
-\langle \delta\Bp^n_{h},\mubf_{h}\rangle +\langle \nabla\delta u^n_{h},\mubf_{h}\rangle  & =R^n_{\lambdabf}\left(\mubf_{h}\right),
\end{cases}\label{eq:weakhnewton}
\end{equation}
where the residual functionals are defined by
\begin{align*}
R^n_{\Bp}\left(\Bq_{h}\right) & \coloneqq\langle -\alpha\Bp_{h}^{n}/|\Bp_{h}^{n}|_{\beta}+\mathbf{\lambdabf}_{h}^{n},\Bq_{h}\rangle ,\\
R^n_{u}\left(v_{h}\right) & \coloneqq\langle f-u_{h}^{n},v_{h}\rangle -\langle \lambdabf_{h}^{n},\nabla v_{h}\rangle ,\\
R^n_{\lambdabf}\left(\mubf_{h}\right) & \coloneqq\langle \Bp_{h}^{n}-\nabla u_{h}^{n},\mubf_{h}\rangle .
\end{align*}
Afterwards, the new solution $\left(\Bp_{h}^{n+1},u_{h}^{n+1},\lambdabf_{h}^{n+1}\right)$
is given by
\begin{equation}
\left(\Bp_{h}^{n+1},u_{h}^{n+1},\lambdabf_{h}^{n+1}\right)=\left(\Bp_{h}^{n},u_{h}^{n},\lambdabf_{h}^{n}\right)+\left(\delta\Bp^n_{h},\delta u^n_{h},\delta\lambdabf^n_{h}\right).\label{eq:updateh}
\end{equation}
Given a $\Br \in \BV_{h}$,  we define the bilinear form $a_\Br(\cdot,\cdot)$ as:
\begin{align*}
a_{\Br}\left((\Bp_{h},u_{h},\lambdabf_{h}),(\Bq_{h},v_{h},\mubf_{h})\right) & \coloneqq\langle \alpha H(\Br)\Bp_{h},\Bq_{h}\rangle -\langle \mathbf{\lambdabf}_{h},\Bq_{h}\rangle +\langle u_{h},v_{h}\rangle \nonumber\\
 & \qquad+\langle \mathbf{\lambdabf}_{h},\nabla v_{h}\rangle -\langle \Bp_{h},\mubf_{h}\rangle +\langle \nabla u_{h},\mubf_{h}\rangle,
\end{align*}
and the residual functional $\ell_h(\cdot )$ as:
\[
\ell(\Bq_{h},v_{h},\mubf_{h})\coloneqq R^n_{\Bp}\left(  \Bq_{h}\right)+R^n_{u}\left(v_{h}\right )  +R^n_{\lambdabf}\left ( \mubf_{h}\right )  .
\]
Write $X_{h}\coloneqq\BV_{h}\times U_{h}\times\BW_{h}$, then the problem \eqref{eq:weakhnewton} can be equivalently written as:
Given $\left(\Bp_{h}^{n},u_{h}^{n},\lambdabf_{h}^{n}\right)\in X_{h}$,
find $\left(\delta\Bp^n_{h},\delta u^n_{h},\delta\lambdabf^n_{h}\right)\in X_{h}$
such that
\begin{equation}
a_{\Bp^n_{h}}\left((\delta\Bp^n_{h},\delta u^n_{h},\delta\lambdabf^n_{h}),(\Bq_{h},v_{h},\mubf_{h})\right)=\ell(\Bq_{h},v_{h},\mubf_{h}),\quad\forall(\Bq_{h},v_{h},\mubf_{h})\in X_{h}.\label{eq:weakh}
\end{equation}
%

%Given the previous
%data $\Br\in X_{h}$, the operator $\mathcal{A}_{h}(\Br)$
%in block forms reads
%\begin{equation}
%\mathcal{A}_{h}(\Br)=\left(\begin{array}{ccc}
%\alpha H(\Br) & 0 & -\mathbf{I}\\
%0 & \mathrm{I} & \nabla^{*}\\
%-\mathbf{I} & \nabla & 0
%\end{array}\right)\colon\quad X_{h}\mapsto X_{h}^{*}, \label{eq:Ar}
%\end{equation}
%where $\nabla^{*}$ is the dual operator of $\nabla$, which is defined by
%\begin{align*}
%\left(\nabla^{*}\lambdabf_h,v_h\right)\coloneqq \left(\lambdabf_h,\nabla v_h\right), \qquad \forall \lambdabf_h \in \BW_h, v_h \in V_h.
%\end{align*}
For a given $\Br\in \BV_{h}$, we introduce the operator $\mathcal{A}_{h}(\Br): X_h \rightarrow X_h^{*}$ defined by
\begin{equation}
\langle\mathcal{A}_{h}(\Br)(\Bp_{h}, u_{h},\lambdabf_{h}),(\Bq_{h},v_{h},\mubf_{h})\rangle \coloneqq a_{\Br}\left((\Bp_{h},u_{h},\lambdabf_{h}),(\Bq_{h},v_{h},\mubf_{h})\right),\,\forall (\Bp^n_{h}, u^n_{h},\lambdabf^n_{h}),\,(\Bq_{h},v_{h},\mubf_{h})\in X_{h}.\label{eq:Ar}
\end{equation}
We remind that  $\langle\cdot,\cdot\rangle$ here refers to the  duality pairing between $X_{h}^{*}$ and $X_{h}$. Roughly speaking, $\mathcal{A}_{h}(\Br)$ is the discrete version of the following operator,
\begin{equation}
\mathcal{A}(\Br)\coloneqq\left(\begin{array}{ccc}
\alpha H(\Br) & 0 & -\mathbf{I}\\
0 & \mathrm{I} & -\nabla\cdot\\
-\mathbf{I} & \nabla & 0
\end{array}\right). \label{eq:cAr}
\end{equation}
We first note the form $\langle\mathcal{A}_{h}(\Br)\cdot,\cdot\rangle$
is symmetric, reflecting the symmetry of the saddle point operator
$\mathcal{A}_{h}(\Br)$. Besides that, for the $n$-th step of
the Newton iteration, we need to solve a linear system with the operator
$\mathcal{A}_{h}(\Bp_{h}^{n})$ as the coefficient matrix. Since the operator $\mathcal{A}_{h}(\Br)$ is indefinite and symmetric, we solve the linear system by the minimum residual method (MINRES) \cite{Mardal2011}. However, as one of the  Krylov space methods, the convergence rate of MINRES depends on the condition number and the spectrum of $\mathcal{A}_{h}(\Br)$.
This motivates
the study of efficient preconditioners for the operator $\mathcal{A}_{h}(\Br)$.
\begin{rem}
Using the inverse estimate in the finite element space $\BV_{h}$,
we have
\begin{equation}
\left\Vert \Bp_{h}\right\Vert _{\infty}\le C_{h}\left\Vert \Bp_{h}\right\Vert,\label{eq:inverse}
\end{equation}
where the embedding constant $C_{h}$ depends on the
mesh size but is finite. Thus, $H(\Bp_{h}^{n})\in\BL^{\infty}(\Omega)$ is uniformly elliptic, c.f \eqref{eq:MatrixH}. Accordingly, problems \eqref{eq:weakhfull} and \eqref{eq:weakhnewton}
are well-defined.
\end{rem}

\section{Well-posedness of the linear system for the Newton algorithm \label{sec:wellposed}}

Now we discuss the well-posedness of scheme \eqref{eq:weakh}, which
is the foundation of the preconditioners we proposed. For this end,
we will use Babu\v{s}ka-Brezzi theory to analyze the mapping properties of the
operator $\mathcal{A}_{h}(\Br)$ for any $\Br  \in \BV_h $.

As discussed in \cite{Mardal2011}, ensuring that the continuity
constants and the inf\textendash sup constants are independent of
the physical parameters and the discretized parameters is crucial to design robust
block preconditioners for solving our problem. Note that the natural bound
on the saddle point operator $\mathcal{A}_{h}(\Br)$ depends on
the vector field $\Br$. Therefore, we equip the space $X_{h}$
with the following $\Br$-dependent inner-product
\begin{equation}
\left((\Bp_h,u_h,\lambdabf_h),(\Bq_h,v_h,\mubf_h)\right) _{X_{r}}\coloneqq\alpha\langle H(\Br)\Bp_h,\Bq_h\rangle +\langle u_h, v_h\rangle+\alpha\langle H(\Br)\nabla u_h,\nabla v_h\rangle +\alpha^{-1}\langle H(\Br)^{-1}\lambdabf_h,\mubf_h\rangle ,
\label{eq:inner}
\end{equation}
and norm
\begin{equation}
\left\Vert (\Bq_h,v_h,\mubf_h)\right\Vert _{X_{r}}^{2}\coloneqq\alpha\langle H(\Br)\Bq_h,\Bq_h\rangle +\left\Vert v_h\right\Vert ^{2}+\alpha\langle H(\Br)\nabla v_h,\nabla v_h\rangle +\alpha^{-1}\langle H(\Br)^{-1}\mubf_h,\mubf_h\rangle .\label{eq:nrom}
\end{equation}
It is clear that the corresponding norm for the dual space $X_{h}^{*}$
depends on $\Br$ and $\alpha$. Thanks to the estimate \eqref{eq:inverse}
and \eqref{eq:MatrixH}, the matrix $H(\Br)$ is uniformly
elliptic and invertible. Hence, the weighted norm \eqref{eq:nrom}
is well-defined in $X_{h}$.

After the introduction of these notations, we are able to show that
the operator $\mathcal{A}_{h}(\Br)$ has the following properties.
\begin{lem}[Boundedness of $\mathcal{A}_{h}(\Br)$]
\label{lem:bound} There is a positive constant $c_{0}$, independent
of $\Br,\alpha,\beta$ and $h$, such that for all $(\Bp_h,u_h,\lambdabf_h)$, $(\Bq_h,v_h,\mubf_h)\in X_{h}$,
we have
\[
\langle\mathcal{A}_{h}(\Br)(\Bp_h,u_h,\lambdabf_h),(\Bq_h,v_h,\mubf_h)\rangle\leq c_{0}\left\Vert (\Bp_h,u_h,\lambdabf_h)\right\Vert _{X_{r}}\left\Vert (\Bq_h,v_h,\mubf_h)\right\Vert _{X_{r}}.
\]
\end{lem}
\begin{proof}
This follows directly from the definition of the operator $\mathcal{A}_{h}(\Br)$
and the definition of $\left\Vert \cdot\right\Vert _{X_{r}}$ with
$c_{0}=2$.
\end{proof}

We define the associated kernel space $Z_{h}\subset\BV_{h}\times U_{h}$
by
\[
Z_{h}=\left\{ (\Bp_{h},u_{h})\in\BV_{h}\times U_{h}\mid\langle \nabla u_{h}-\Bp_{h},\mubf_{h}\rangle =0, \quad\forall\mubf_{h}\in\BW_{h}\right\} .
\]
Note that $\nabla u_{h}\in\BW_{h}$ and $\boldsymbol{p}_{h}\in\BW_{h}$,
thus taking $\mubf_{h}=\nabla u_{h}-\Bp_{h}$ gives $\Bp_{h}=\nabla u_{h}$.
Consequently, the discrete kernel space $Z_{h}$ is further portrayed
as
\[
Z_{h}=\left\{ (\Bp_{h},u_{h})\in\BV_{h}\times U_{h}\mid\Bp_{h}=\nabla u_{h}\right\} .
\]

\begin{lem}[Coercivity on the kernel space]
\label{lem:cor} There is a positive constant $c_{1}$, independent
of $\Br,\alpha,\beta$ and $h$, such that for all $(\Bp_h,u_h)\in Z_{h}$,
we have
\[
\langle\mathcal{A}_{h}(\Br)(\Bp_h,u_h,0),(\Bp_h,u_h,0)\rangle\geq c_{1}\left\Vert (\Bp_h,u_h,0)\right\Vert _{X_{r}}^{2}.
\]
\end{lem}
\begin{proof}
Just observe that on $Z_{h}$
\begin{align*}
\langle\mathcal{A}_{h}(\Br)(\Bp_h,u_h,0),(\Bp_h,u_h,0)\rangle & =\langle \alpha H(\Br)\Bp_h,\Bp_h\rangle +\left\Vert u_h\right\Vert ^{2}\\
 & =\frac{1}{2}\langle \alpha H(\Br)\Bp_h,\Bp_h\rangle +\left\Vert u_h\right\Vert ^{2}+\frac{1}{2}\langle \alpha H(\Br)\nabla u_h,\nabla u_h\rangle ,
\end{align*}
 and that
\[
\left\Vert (\Bp_h,u_h,0)\right\Vert _{X_{r}}^{2}=\langle \alpha H(\Br)\Bp_h,\Bp_h\rangle +\left\Vert u_h\right\Vert ^{2}+\alpha\langle H(\Br)\nabla u_h,\nabla u_h\rangle .
\]
 Then the desired inequality holds with $c_{1}=1/2$.
\end{proof}
\begin{lem}[Inf-sup condition]
\label{lem:inf-sup} There is a positive constant $c_{2}$, independent
of $\Br,\alpha,\beta$ and $h$, such that
\[
\sup_{(\Bq_h,v_h)\in\BV_{h}\times U_{h}}\frac{\langle\mathcal{A}_{h}(\Br)(0,0,\lambdabf_h),(\Bq_h,v_h,0)\rangle}{\left\Vert (\Bq_h,v_h,0)\right\Vert _{X_{r}}}\geq c_{2}\left\Vert (0,0,\lambdabf_h)\right\Vert _{X_{r}},\quad\forall\lambdabf\in\BW_{h}.
\]
\end{lem}
\begin{proof}
We have that
\[
\begin{aligned}\sup_{(\Bq_h,v_h)\in\BV_{h}\times U_{h}}\frac{\langle\mathcal{A}_{h}({\Br})(0,0,\lambdabf_h),(\Bq_h,v_h,0)\rangle}{\left\Vert (\Bq_h,v_h,0)\right\Vert _{X_{r}}} & \geq\sup_{(\Bq_h,0)\in\BV_{h}\times U_{h}}\frac{-\langle \lambdabf_h,\Bq_h\rangle }{\left (  \alpha\langle H({\Br})\Bq_h,\Bq_h\rangle \right){}^{1/2}}\\
 & \geq\left(\alpha^{-1}\langle H(\Br)^{-1}\lambdabf_h,\lambdabf_h\rangle \right)^{1/2},
\end{aligned}
\]
 where the last inequality follows by taking $\Bq_h=-\alpha^{-1}H({\Br})^{-1}\lambdabf_h$.
Since
\[
\left\Vert (0,0,\lambdabf_h)\right\Vert _{X_{r}}=\left(\alpha^{-1}\langle H({\Br})^{-1}\lambdabf_h,\lambdabf_h\rangle \right)^{1/2}.
\]
This shows that the desired inequality holds with $c_{2}=1$.
\end{proof}

With the three lemmas above, we obtain the main result of this section.
\begin{thm}
\label{thm:wellposed}At each iteration for the Newton updating, the discretized
problem \eqref{eq:weakhnewton} is well-posed.
\end{thm}
\begin{proof}
By similar arguments of Theorem 5.2 in \cite{Mardal2011}, it is
straight-forward to reach the conclusion.
\end{proof}

Using adequately weighted spaces, we have that $\mathcal{A}_{h}({\Br})$
is an isomorphism from $X_{h}$ to $X_{h}^{*}$ such that $\left\Vert \mathcal{A}_{h}(\Br)\right\Vert _{\mathcal{L}\left(X_{h},X_{h}^{*}\right)}$
and $\left\Vert \left(\mathcal{A}_{h}(\Br)\right)^{-1}\right\Vert _{\mathcal{L}\left(X_{h}^{*},X_{h}\right)}$
are bounded independently of mesh sizes and the parameters. That is
\begin{equation}
\left\Vert \mathcal{A}_{h}(\Br)\right\Vert _{\mathcal{L}\left(X_{h},X_{h}^{*}\right)}\le C,\quad\left\Vert \left(\mathcal{A}_{h}(\Br)\right)^{-1}\right\Vert _{\mathcal{L}\left(X_{h}^{*},X_{h}\right)}\le c^{-1},\label{eq:A_cons}
\end{equation}
 where the constants $C$ and $c$ only depend on $c_{0},c_{1},c_{2}$.

\section{Robust preconditioners\label{sec:precond}}

In this section, we develop and analyze robust preconditioners.  Let $\mathcal{B}_h(\Br)$ be the \textquotedbl Riesz-operator\textquotedbl{}
mapping from $X_{h}^{*}$ to $X_{h}$, which is induced by the weighted norm $\Vert\cdot\Vert_{X_r}$, for given $(\Bp_{h}, u_{h},\lambdabf_{h})\in X_{h}^{*}$:
$$
\left(\mathcal{B}_h(\Br) (\Bp_{h}, u_{h},\lambdabf_{h}),(\Bq_{h},v_{h},\mubf_{h})\right)_{X_r}\coloneqq\langle (\Bp_{h}, u_{h},\lambdabf_{h}), (\Bq_{h},v_{h},\mubf_{h})\rangle, \quad (\Bq_{h},v_{h},\mubf_{h}) \in X_{h}.
$$
By using \eqref{eq:inner} and \eqref{eq:nrom}, we can see that it takes the following explicit form,
\begin{equation}
\mathcal{B}_{h}(\Br)=\left(\begin{array}{ccc}
\alpha H({\Br}) & 0 & 0\\
0 & S_{\Br,h} & 0\\
0 & 0 & \alpha^{-1}H({\Br})^{-1}
\end{array}\right)^{-1},\label{eq:Br}
\end{equation}
%\begin{equation}
%	\langle\left(\mathcal{B}_{h}(\Br)\right)^{-1}(\Bp_{h},u_{h},\lambdabf_{h}),(\Bq_{h},v_{h},\mubf_{h})\rangle \coloneqq\langle \alpha H(\Bp_{h}^{n})\Bp_{h},\Bq_{h}\rangle ??? + \langle S_{\Br}u_{h},v_{h}\rangle+\langle  \alpha^{-1}H({\Br})^{-1}\lambdabf_{h},\mubf_{h}\rangle ,\label{eq:Br}
%\end{equation}
where the operator $S_{\Br,h}\colon\,V_{h}\mapsto V_{h}^{*}$ is defined by
\[
\langle S_{\Br,h}u_{h},v_{h}\rangle \coloneqq\langle u_{h},v_{h}\rangle +\langle \alpha H({\Br})\nabla u_{h},\nabla v_{h}\rangle ,  \qquad \forall u_h, v_h\in V_h.
\]
In fact, $S_{\Br,h}$ is the finite element discretization for the operator $\rmI - \nabla\cdot\left(H\left(\Br\right)\nabla\right)$ with Neumann boundary conditions.
Note that the matrix $H(\Br)$ is symmetric and uniformly elliptic, the second block, $S_{\Br,h}$,
is invertible. Following the operator preconditioning framework \cite{Mardal2011},
the operator $\mathcal{B}_{h}(\Br)$ is proposed as the preconditioner
for $\mathcal{A}_{h}(\Br)$.

In the following we  estimate the condition number of the preconditioned operator $\mathcal{B}_{h}(\Br)\mathcal{A}_{h}(\Br)$ using the following formula:
\begin{equation}\label{eq:kappa} \kappa\left(\mathcal{B}_{h}(\Br)\mathcal{A}_{h}(\Br)\right)=\left\Vert \mathcal{B}_{h}(\Br)\mathcal{A}_{h}(\Br)\right\Vert {}_{\mathcal{L}\left(X_{h},X_{h}\right)}\left\Vert \left(\mathcal{B}_{h}(\Br)\mathcal{A}_{h}(\Br)\right)^{-1}\right\Vert _{\mathcal{L}\left(X_{h},X_{h}\right)}.
\end{equation}
\begin{thm}
\label{thm:condexact}
The condition number $\kappa\left(\mathcal{B}_{h}(\Br)\mathcal{A}_{h}(\Br)\right)$
has a uniform bound, independent
of $\Br,\alpha,\beta$ and $h$, in the sense that
\[
1\le\kappa\left(\mathcal{B}_{h}(\Br)\mathcal{A}_{h}(\Br)\right)\le C/c,
\]
where $C$ and $c$ are defined by \eqref{eq:A_cons}.
\end{thm}
\begin{proof}
We note that the operator $\mathcal{B}_{h}(\Br)$ has the property
that
\[
\left\Vert \mathcal{B}_{h}({\Br})\right\Vert {}_{\mathcal{L}\left(X_{h}^{*},X_{h}\right)}=\left\Vert \left(\mathcal{B}_{h}({\Br})\right)^{-1}\right\Vert _{\mathcal{L}\left(X_{h},X_{h}^{*}\right)}=1.
\]
 Then we have
\[
\left\Vert \mathcal{B}_{h}(\Br)\mathcal{A}_{h}(\Br)\right\Vert {}_{\mathcal{L}\left(X_{h},X_{h}\right)}=\left\Vert \mathcal{A}_{h}(\Br)\right\Vert {}_{\mathcal{L}\left(X_{h},X_{h}^{*}\right)}\leq C,
\]
 and
\[
\left\Vert \left(\mathcal{B}_{h}(\Br)\mathcal{A}_{h}(\Br)\right)^{-1}\right\Vert _{\mathcal{L}\left(X_{h},X_{h}\right)}\leq\left\Vert \left(\mathcal{A}_{h}(\Br)\right)^{-1}\right\Vert _{\mathcal{L}\left(X_{h}^{*},X_{h}\right)}\left\Vert \left(\mathcal{B}_{h}(\Br)\right)^{-1}\right\Vert _{\mathcal{L}\left(X_{h},X_{h}^{*}\right)}\leq c^{-1}.
\]
We get the desired result using \eqref{eq:kappa}.
\end{proof}

This theorem suggests that $\mathcal{B}_{h}(\Br)$ is a ``good''
preconditioner for $\mathcal{A}_{h}(\Br)$ in the sense that the
preconditioned operator has a condition number which is independent of
the mesh size ($h$), the penalization parameter ($\alpha$), the regularization
parameter ($\beta$) and the iterative step ($n$). We are using MINRES to solve the preconditioned linear system  at each Newton iteration.  We have the following convergence result for the MINRES.
\begin{thm}
\noindent \label{thm:conexact}If $x^{0}$ is the the initial value,
$x^{m}$ is the $m$-th iteration of MINRES method and $x$ is the
exact solution, then there exists a constant $\delta\in(0,1)$, only
depending on the condition number $\kappa\left(\mathcal{B}_{h}(\Br)\mathcal{A}_{h}(\Br)\right)$,
such that
\[
\left\langle \mathcal{B}_{h}(\Br)\mathcal{A}_{h}(\Br)\left(x-x^{m}\right),\mathcal{A}_{h}(\Br)\left(x-x^{m}\right)\right\rangle ^{\frac{1}{2}}\le2\delta^{m}\left\langle \mathcal{B}_{h}(\Br)\mathcal{A}_{h}(\Br)\left(x-x^{0}\right),\mathcal{A}_{h}(\Br)\left(x-x^{0}\right)\right\rangle ^{\frac{1}{2}}.
\]
Moreover, an estimate leads to
\[
\delta=\frac{\kappa\left(\mathcal{B}_{h}(\Br)\mathcal{A}_{h}(\Br)\right)-1}{\kappa\left(\mathcal{B}_{h}(\Br)\mathcal{A}_{h}(\Br)\right)+1}\le\frac{C-c}{C+c}.
\]
\end{thm}

\begin{proof} Since the operator $\mathcal{B}_{h}(\Br)\in \mathcal{L}(X^*_h,X_h)$ is symmetric and positive definite, $\left\langle \left(\mathcal{B}_{h}(\Br)\right)^{-1} \cdot, \cdot\right\rangle$ defines an inner product on $X_h$. Furthermore, it is easy to see that the preconditioned operator $\mathcal{B}_{h}(\Br)\mathcal{A}_{h}(\Br)\in \mathcal{L}(X_h,X_h)$ is symmetric in this inner product.
So we shall use MINRES method which is defined with respect to the inner product  $\left\langle \left(\mathcal{B}_{h}(\Br)\right)^{-1} \cdot, \cdot\right\rangle$. The estimate is obtained by applying Theorem 2.2 in \cite{Mardal2011}.
\end{proof}

% From \eqref{eq:Br}, we see that the preconditioner $\mathcal{B}_{h}(\Br)$ has a diagonal block structure. The diagonal entries involving $H(\Br)$ are easy to invert and only needs elementary operations. It is remarked that the structures of the blocks $\alpha H({\Br})$
% and $\alpha^{-1}H({\Br})^{-1}$ are element-by-element, thus they
% can be inverted exactly and easily.
The inverse of $S_{\Br,h}$ needs to solve an elliptic  linear problem. It could be costly to solve this elliptic problem, especially for three dimensional prolems:
%%
%%
%The previous theory establishes a robust estimate when direct methods are used
%for the block solvers. However, obtaining sufficiently resolved numerical
%solutions requires using small mesh sizes, which translates into solving
%correspondingly larger linear systems. For these large-scale computations,
% particularly in three dimensions, using direct methods for the block
% solvers becomes prohibitively expensive and time-consuming. For practical
% computations, exact solvers are replaced with fast, inexact methods
% that ideally maintain robust convergence. To ensure that the proven
% robust spectral bound holds, the exact block solvers representing
% the underlying operators are often replaced with robust operators.
%
 Thus, we consider the following
inexact preconditioner
\begin{equation}
\widetilde{\mathcal{B}}_{h}(\Br)=\left(\begin{array}{ccc}
\alpha H({\Br}) & 0 & 0\\
0 & \widetilde{S}_{\Br,h}^{-1} & 0\\
0 & 0 & \alpha^{-1}H({\Br})^{-1}
\end{array}\right)^{-1}.\label{eq:inexact}
\end{equation}
%\begin{equation}
%	\langle\left(\widetilde{\mathcal{B}}_{h}(\Br)\right)^{-1}(\Bp_{h},u_{h},\lambdabf_{h}),(\Bq_{h},v_{h},\mubf_{h})\rangle \coloneqq\left(\alpha H(\Bp_{h}^{n})\Bp_{h},\Bq_{h}\right)+\langle \widetilde{S_{\Br}}^{-1}u_{h},v_{h}\rangle+\left( \alpha^{-1}H({\Br})^{-1}\lambdabf_{h},\mubf_{h}\right),\label{eq:inexact}
%\end{equation}
Here, $\widetilde{S}_{\Br,h}$ is an operator that is spectrally equivalent to the action of the inverse
of the block $S_{\Br,h}$ in the following sense:
\begin{equation}
c_{1,s}\langle S_{\Br,h}^{-1}v_h,v_h\rangle\leq\langle\widetilde{S}_{\Br,h}v_h,v_h\rangle\leq c_{2,s}\langle S_{\Br,h}^{-1}v_h,v_h\rangle\quad \forall v_h\in V_h,\label{eq:spec}
\end{equation}
where the constants $c_{1,s}$ and $c_{2,s}$ are independent of the
mesh size, the penalization parameter, the regularization parameter
and the iterative step. In the implementation, $\widetilde{S}_{\Br,h}$ is obtained
by using standard multigrid method, algebraic multigrid method (AMG) or domain
decomposition method. Based on the results in \cite{Briggs2000,Hackbusch2016,Xu2017}, we think the estimate \eqref{eq:spec} will hold if we use AMG to produce
$\tilde S_{\Br,h}$, and our experiments seem to confirm this. Thus, the proposed preconditioner is very cheap and easy-to-implement.

When using $\widetilde{\mathcal{B}}_{h}(\Br)$ as the preconditioner, we have the following results regarding the convergence estimate
of the corresponding preconditioned MINRES method.
\begin{thm}
\noindent Assume that \eqref{eq:spec} is satisfied, then we have
\[
1\le\kappa\left(\widetilde{\mathcal{B}}_{h}({\Br})\mathcal{A}_{h}({\Br})\right)\le\frac{C\hat{c}_{2}}{c\hat{c}_{1}},
\]

\noindent where $\hat{c}_{1}=\min\left(c_{1,s},1\right)$ and $\hat{c}_{2}=\max\left(c_{2,s},1\right)$.
Moreover, if $x^{0}$ is the the initial value, $x^{m}$ is the $m$-th
iteration of MINRES method and $x$ is the exact solution, then there
exists a constant $\delta\in(0,1)$, only depending on the condition
number $\kappa\left(\widetilde{\mathcal{B}}_{h}({\Br})\mathcal{A}_{h}({\Br})\right)$,
such that
\[
\left\langle \widetilde{\mathcal{B}}_{h}({\Br})\mathcal{A}_{h}({\Br})\left(x-x^{m}\right),\mathcal{A}_{h}(\Br)\left(x-x^{m}\right)\right\rangle ^{\frac{1}{2}}\le2\delta^{m}\left\langle \widetilde{\mathcal{B}}_{h}({\Br})\mathcal{A}_{h}({\Br})\left(x-x^{0}\right),\mathcal{A}_{h}(\Br)\left(x-x^{0}\right)\right\rangle ^{\frac{1}{2}}.
\]
 Furthermore, an estimate leads to
\[
\delta=\frac{\kappa\left(\widetilde{\mathcal{B}}_{h}({\Br})\mathcal{A}_{h}({\Br})\right)-1}{\kappa\left(\widetilde{\mathcal{B}}_{h}({\Br})\mathcal{A}_{h}({\Br})\right)+1}\le\frac{\hat{c}_{2}C-\hat{c}_{1}c}{\hat{c}_{2}C+\hat{c}_{1}c}.
\]
\end{thm}
\begin{proof}
We only need to estimate the condition number of the preconditioned
operator $\widetilde{\mathcal{B}}_{h}({\Br})\mathcal{A}_{h}({\Br})$, the
result of the convergence of preconditioned MINRES method is obvious. From
\eqref{eq:spec}, we can see that $\widetilde{\mathcal{B}}_{h}({\Br})$
satisfies
\begin{equation}
\hat{c}_{1}\left(x,x\right)_{\mathcal{B}_{h}(\Br)}\le(x,x)_{\widetilde{\mathcal{B}}_{h}(\Br)}\le\hat{c}_{2}\left(x,x\right)_{\mathcal{B}_{h}(\Br)}.\label{eq:ineq}
\end{equation}
This implies that
\[
\left\Vert \widetilde{\mathcal{B}}_{h}(\Br)\left(\mathcal{B}_{h}(\Br)\right)^{-1}\right\Vert {}_{\mathcal{L}\left(X_{h},X_{h}\right)}\le\hat{c}_{2},\quad\left\Vert \left(\widetilde{\mathcal{B}}_{h}(\Br)\left(\mathcal{B}_{h}(\Br)\right)^{-1}\right)^{-1}\right\Vert {}_{\mathcal{L}\left(X_{h},X_{h}\right)}\le\hat{c}_{1}^{-1}.
\]
Similar to the proof of Theorem \ref{thm:condexact}, we have
\begin{align*}
 & \left\Vert \widetilde{\mathcal{B}}_{h}(\Br)\mathcal{A}_{h}(\Br)\right\Vert {}_{\mathcal{L}\left(X_{h},X_{h}\right)}=\left\Vert \widetilde{\mathcal{B}}_{h}(\Br)\left(\mathcal{B}_{h}(\Br)\right)^{-1}\mathcal{B}_{h}(\Br)\mathcal{A}_{h}(\Br)\right\Vert {}_{\mathcal{L}\left(X_{h},X_{h}\right)}\\
 & \quad\le\left\Vert \widetilde{\mathcal{B}}_{h}(\Br)\left(\mathcal{B}_{h}(\Br)\right)^{-1}\right\Vert {}_{\mathcal{L}\left(X_{h},X_{h}\right)}\|\mathcal{B}_{h}(\Br)\mathcal{A}_{h}(\Br)\|_{\mathcal{L}\left(X_{h},X_{h}\right)}\leq C\hat{c}_{2},
\end{align*}
 and
\begin{align*}
 & \left\Vert \left(\widetilde{\mathcal{B}}_{h}(\Br)\mathcal{A}_{h}(\Br)\right)^{-1}\right\Vert _{\mathcal{L}\left(X_{h},X_{h}\right)}=\left\Vert \left(\widetilde{\mathcal{B}}_{h}(\Br)\left(\mathcal{B}_{h}(\Br)\right)^{-1}\mathcal{B}_{h}(\Br)\mathcal{A}_{h}(\Br)\right)^{-1}\right\Vert {}_{\mathcal{L}\left(X_{h},X_{h}\right)}\\
 & \quad\leq\left\Vert \left(\mathcal{B}_{h}(\Br)\mathcal{A}_{h}(\Br)\right)^{-1}\right\Vert _{\mathcal{L}\left(X_{h},X_{h}\right)}\left\Vert \left(\widetilde{\mathcal{B}}_{h}(\Br)\left(\mathcal{B}_{h}(\Br)\right)^{-1}\right)^{-1}\right\Vert {}_{\mathcal{L}\left(X_{h},X_{h}\right)}\leq c^{-1}\hat{c}_{1}^{-1}.
\end{align*}
Therefore, we get the estimate of the condition number.
\end{proof}

As a consequence, the exact methods for the block solvers in the preconditioner $\mathcal{B}_{h}(\Br)$
can be replaced by inexact methods and still maintain the desired
properties. So far, we have gotten a robust and effective solver when solving the linearized systems. This makes it very cheap
to solve the system for each iteration for Newton updating. Together,
it yields a second-order Newton scheme with robust and optimal preconditioners.

\section{Numerical experiments\label{sec:num}}

In this section, we present some numerical experiments to verify the
convergence rate of the finite element approximation to  our model, and to demonstrate the robustness of the preconditioners. A set of
two-dimensional examples are reported below. All codes were
written in MATLAB based on the open-source finite element library
$i$FEM \cite{Chen2008}.

For comparison, we also implement the following fixed point method, which is
also known as Picard method. For this method, given the $n$-th iteration
$\left(\Bp_{h}^{n},u_{h}^{n},\lambdabf_{h}^{n}\right)\in\BV_{h}\times U_{h}\times\BW_{h}$,
$\left(\Bp_{h}^{n+1},u_{h}^{n+1},\lambdabf_{h}^{n+1}\right)\in\BV_{h}\times U_{h}\times\BW_{h}$
is solved by
\begin{equation}
	\begin{cases}
		\langle\alpha\Bp_{h}^{n+1}/\left|\Bp_{h}^{n}\right|_{\beta},\Bq_{h}\rangle-\langle\mathbf{\lambdabf}_{h}^{n+1},\Bq_{h}\rangle & =0\quad\forall\Bq_{h}\in\BV_{h},\\
		\langle u_{h}^{n+1},v_{h}\rangle+\langle\mathbf{\lambdabf}_{h}^{n+1},\nabla v_{h}\rangle-\langle f,v_{h}\rangle & =0\quad\forall v_{h}\in U_{h},\\
		-\langle\Bp_{h}^{n+1},\mubf_{h}\rangle+\langle\nabla u_{h}^{n+1},\mubf_{h}\rangle & =0\quad\forall\mubf_{h}\in\BW_{h}.
	\end{cases}\label{eq:weakhpiacrd}
\end{equation}
With the notation in Section \ref{sec:fem},
its residual form reads: Find $\left(\delta\Bp^n_{h},\delta u^n_{h},\delta\lambdabf^n_{h}\right)\in\BV_{h}\times U_{h}\times\BW_{h}$
such that for any $(\Bq_{h},v_{h},\mubf_{h})\in\BV_{h}\times U_{h}\times\BW_{h}$,
\begin{equation}
	\begin{cases}
		\langle\alpha\widehat{H}(\Bp_{h}^{n})\delta\Bp^n_{h},\Bq_{h}\rangle-\langle\delta\lambdabf^n_{h},\Bq_{h}\rangle & =R^n_{\Bp}\left(\Bq_{h}\right),\\
		\langle\delta u^n_{h},v_{h}\rangle+\langle\delta\lambdabf^n_{h},\nabla v_{h}\rangle & =R^n_{u}\left(v_{h}\right),\\
		-\langle\delta\Bp^n_{h},\mubf_{h}\rangle+\langle\nabla\delta u^n_{h},\mubf_{h}\rangle & =R^n_{\lambdabf}\left(\mubf_{h}\right),
	\end{cases}\label{eq:weakhpicardres}
\end{equation}
where $\widehat{H}(\Bp_{h}^{n})\coloneqq1/\left|\Bp_{h}^{n}\right|_{\beta}$.
Thus, the new solution $\left(\Bp_{h}^{n+1},u_{h}^{n+1},\lambdabf_{h}^{n+1}\right)$
is given by
\begin{equation}
	\left(\Bp_{h}^{n+1},u_{h}^{n+1},\lambdabf_{h}^{n+1}\right)=\left(\Bp_{h}^{n},u_{h}^{n},\lambdabf_{h}^{n}\right)+\left(\delta\Bp^n_{h},\delta u^n_{h},\delta\lambdabf^n_{h}\right).\label{eq:updatehp}
\end{equation}
Clearly, the difference of Newton method and Picard method lies on
the matrix $H(\Br)$ and the scalar $\widehat{H}(\Br)$. Note that $\widehat{H}(\Br)$ is bounded above by $1/\beta$ and below by zero. Hence, the theory for Newton method in Sections \ref{sec:wellposed}-\ref{sec:precond} also applies to the Picard method as long as we replace $H(\Br)$ by $\widehat{H}(\Br)$ in $\mathcal{A}_{h}(\Br)$, $\mathcal{B}_{h}(\Br)$ and $\widetilde{\mathcal{B}}_{h}(\Br)$. For convenience,
we still use the notation $\mathcal{A}_{h}(\Br)$, $\mathcal{B}_{h}(\Br)$ and $\widetilde{\mathcal{B}}_{h}(\Br)$ for Picard method. From \eqref{eq:weakhnewton}, \eqref{eq:weakhpicardres}, \eqref{eq:Br} and \eqref{eq:inexact}, we know that the costs of per iteration for Picard method and Newton method are almost the same. Thus, for convenience, only the iteration numbers needed by both methods are used for comparison in this section.

\subsection{Implementation of block preconditioners}

First of all, we discuss some implementation details of the proposed
block preconditioners. To solve the linear system obtained from the
finite element discretization, we use the MINRES method as an outer
iterative solver, with the tolerance for the relative residual in
the energy norm set to $\varepsilon=10^{-10}$. The block preconditioners designed
in Section \ref{sec:precond} are used to accelerate the convergence
rate of MINRES. In the following, we implement both exact and inexact
inner solvers, $\mathcal{B}_{h}(\Br)$ and $\widetilde{\mathcal{B}}_{h}(\Br)$.
As we know, inverting the proposed block preconditioners ends up with
inverting diagonal blocks. Therefore, the main difference in implementation
is how to invert the second diagonal block. For the exact preconditioner
$\mathcal{B}_{h}(\Br)$ , we call direct solvers implemented in
MATLAB. While for the inexact preconditioner $\widetilde{\mathcal{B}}_{h}(\Br)$,
we mainly call AMG preconditioned conjugate gradient (PCG) method to
define the operator $\widetilde{S}_{\Br,h}$. The tolerance of PCG in terms of $l_{2}$\textendash norm
of the relative residual is $\varepsilon_{0}=10^{-3}$. The AMG we used is the classical algebraic multigrid method with Ruge-Stuben coarsening and standard interpolation \cite{Chen2008,Xu2017}.

Without specifications, the initial guess $\left(\Bp_{h}^{0},u_{h}^{0},\lambdabf_{h}^{0}\right)$
is taken to be the zero solution and the relative tolerances are set
by $10^{-6}$ for the nonlinear iteration. Here the maximal iteration
number of the MINRES solver is set by $N=200$. For the sake of convenience,
we denote $N_{{\rm Newton}}$ by the number of Newton\textquoteright s
iterations,  $N_{{\rm Picard}}$ by the number of Picard\textquoteright s
iterations and $N_{{\rm MINRES}}$ by the average number of preconditioned
MINRES iterations for solving the linearized problem.

To ensure the
global convergence of the Newton algorithm, we introduce an additional
damping parameter $\theta$. Let $\boldsymbol{b}^{n}$ be the vector
representation of the right hand of \eqref{eq:weakh} at $n$-th iteration,
we use the classical backtracking line search method, which selects as step length $\theta^n$ to be the first number in the sequence of $\left\{ 1/2^{k}\right\} _{k=0}^{\infty}$
that satisfies the following criterion
\[
\left\Vert \boldsymbol{b}^{k+1}\right\Vert _{l^{2}}\le\left(1-\sigma\theta_{k}\right)\left\Vert \boldsymbol{b}^{n}\right\Vert _{l^{2}},
\]
where $\sigma$ is chosen as $10^{-4}$. Thus, the damped Newton-updating
is given by
\begin{equation}
	\left(\Bp_{h}^{n+1},u_{h}^{n+1},\lambdabf_{h}^{n+1}\right)=\left(\Bp_{h}^{n},u_{h}^{n},\lambdabf_{h}^{n}\right)+\theta^n\left(\delta\Bp_{h}^n,\delta u_{h}^n,\delta\lambdabf_{h}^n\right).\label{eq:NewtonupdateDamped}
\end{equation}

\subsection{Numerical results} This subsection is to report on some numerical experiments.
\begin{example}
\label{exa:smooth}This example is to test the convergence rate of
finite element solutions and the robustness of the preconditioners
for some smooth problems. The computational domain $\Omega$ is set as $(0,1)^{2}$.
The function $f$ is chosen so that the exact solutions are given
by
\begin{align*}
u & =\cos\left(\pi x\right)\cos\left(\pi y\right),\\
\Bp & =-\pi\left(\sin\left(\pi x\right)\cos\left(\pi y\right),\cos\left(\pi x\right)\sin\left(\pi y\right)\right)^{\top},\\
\boldsymbol{\lambdabf} & =-\frac{\pi\alpha}{\left|\Bp\right|_{\beta}}\left(\sin\left(\pi x\right)\cos\left(\pi y\right),\cos\left(\pi x\right)\sin\left(\pi y\right)\right)^{\top}.
\end{align*}
\end{example}

We first carry out the numerical experiment with $\alpha=\beta=1$. Table
\ref{tab:DOFs2D} shows the information for the meshes and the number
of degrees of freedom which we are using. Based on the results
shown in Table \ref{tab:Order2D}, we find that the convergence
rates for ($\Bp_{h},u_{h},\lambdabf_{h}$) are given by
\[
\begin{array}{ll}
||\Bp-\Bp_{h}||_{0}\sim\mathcal{O}\left(h\right),\quad & ||\lambdabf-\lambdabf_{h}||_{0}\sim\mathcal{O}\left(h\right),\\
||u-u_{h}||_{1}\sim\mathcal{O}\left(h\right),\quad & ||u-u_{h}||_{0}\sim\mathcal{O}\left(h^{2}\right).
\end{array}
\]
Remember that we are using the piecewise constant finite elements
for discretizing $\Bp$ and $\lambdabf$, the first-order Lagrange
finite elements for discretizing $u$. This means that expected optimal
convergence rates are obtained for all variables.

\begin{table}
\caption{The mesh sizes and the numbers of DOFs.\label{tab:DOFs2D}}

\centering{}%
\begin{tabular}{|c|c|c|c|c|c|}
\hline
Mesh & $h$ & DOFs for $\Bp_{h}$ & DOFs for $u_{h}$ & DOFs for $\lambdabf_{h}$ & Total DOFs\\
\hline
$\mathcal{T}_{1}$ & 6.25e-02 & 1,024 & 289 & 1,024 & 2,337\\
\hline
$\mathcal{T}_{2}$ & 3.13e-02 & 4,096 & 1,089 & 4,096 & 9,281\\
\hline
$\mathcal{T}_{3}$ & 1.56e-02 & 16,384 & 4,225 & 16,384 & 36,993\\
\hline
$\mathcal{T}_{4}$ & 7.81e-03 & 65,536 & 16,641 & 65,536 & 147,713\\
\hline
\end{tabular}
\end{table}

\begin{table}
\begin{centering}
\caption{Errors and convergence rates for $\left(\Bp_{h},u_{h},\lambdabf_{h}\right)$
(Example \ref{exa:smooth}).\label{tab:Order2D}}
\par\end{centering}
\begin{centering}
\begin{tabular}{|c|c|c|c|c|}
\hline
$h$ & $||\boldsymbol{p}-\boldsymbol{p}_{h}||_{0}$ & Order & $||\boldsymbol{\lambdabf}-\boldsymbol{\lambdabf}_{h}||_{0}$ & Order\\
\hline
6.25e-02 & 2.17585e-01 & \textemdash{} & 8.95410e-02 & \textemdash{}\\
\hline
3.13e-02 & 1.08967e-01 & 1.00 & 4.52978e-02 & 0.98\\
\hline
1.56e-02 & 5.45105e-02 & 1.00 & 2.27351e-02 & 1.00\\
\hline
7.81e-03 & 2.72596e-02 & 1.00 & 1.13809e-02 & 1.00\\
\hline
\end{tabular}
\par\end{centering}
\medskip{}

\centering{}%
\begin{tabular}{|c|c|c|c|c|}
\hline
$h$ & $||u-u_{h}||_{1}$ & Order & $||u-u_{h}||_{0}$ & Order\\
\hline
6.25e-02 & 2.17595e-01 & \textemdash{} & 7.97886e-03 & \textemdash{}\\
\hline
3.13e-02 & 1.08968e-01 & 1.00 & 2.02665e-03 & 1.98\\
\hline
1.56e-02 & 5.45107e-02 & 1.00 & 5.12786e-04 & 1.98\\
\hline
7.81e-03 & 2.72596e-02 & 1.00 & 1.32618e-04 & 1.95\\
\hline
\end{tabular}
\end{table}

Table \ref{tab:meshrobust2D} shows iteration counts for Newton method with the block preconditioners $\mathcal{B}_{h}(\Br)$ and $\widetilde{\mathcal{B}}_{h}(\Br)$
for different mesh sizes. We see from the relatively consistent iteration
counts that both the exact and inexact preconditioners are robust
with respect to the mesh size. This demonstrates the optimality of
the linear solver and the efficiency of the preconditioners. Compared
with using the exact block preconditioners, using the inexact one
results in a slight degradation in performance, but nothing significant and negligible.
%Thus, the inexact preconditioners can be able to be solved with less strict tolerance, reducing the computational cost.

For comparisons, we present the corresponding results for Picard method in Table \ref{tab:meshrobust2DPicard}. Observe that the linear iteration numbers of MINRES method for these two methods are almost the same but the nonlinear iteration numbers of Newton method are much less than those of Picard method. Using mesh $\mathcal{T}_4$, we display the convergence histories of Picard method and Newton method in Figure \ref{fig:ResSmooth}. We see that Newton method shows a dramatic faster convergence than the Picard method.

\begin{table}
\centering{}\caption{Iteration counts for Newton method with the block preconditioners $\mathcal{B}_{h}(\Br)$
and $\widetilde{\mathcal{B}}_{h}(\Br)$ (Example \ref{exa:smooth}).\label{tab:meshrobust2D}}
\begin{tabular}{|c|c|c|}
\hline
\multirow{2}{*}{mesh} & \multicolumn{1}{c|}{$\mathcal{B}_{h}(\Br)$} & \multicolumn{1}{c|}{$\widetilde{\mathcal{B}}_{h}(\Br)$}\\
\cline{2-3} \cline{3-3}
 & $N_{{\rm Newton}}(N_{{\rm MINRES}})$ & $N_{{\rm Newton}}(N_{{\rm MINRES}})$\\
\hline
$\mathcal{T}_{1}$ & 5(21) & 5(26)\\
\hline
$\mathcal{T}_{2}$ & 5(20) & 5(25)\\
\hline
$\mathcal{T}_{3}$ & 5(20) & 5(25)\\
\hline
$\mathcal{T}_{4}$ & 5(19) & 5(25)\\
\hline
\end{tabular}
\end{table}

\begin{table}
	\centering{}\caption{Iteration counts for Picard method with the block preconditioners $\mathcal{B}_{h}(\Br)$
		and $\widetilde{\mathcal{B}}_{h}(\Br)$ (Example \ref{exa:smooth}).\label{tab:meshrobust2DPicard}}
	\begin{tabular}{|c|c|c|}
		\hline
		\multirow{2}{*}{mesh} & \multicolumn{1}{c|}{$\mathcal{B}_{h}(\Br)$} & \multicolumn{1}{c|}{$\widetilde{\mathcal{B}}_{h}(\Br)$}\\
		\cline{2-3} \cline{3-3}
		& $N_{{\rm Picard}}(N_{{\rm MINRES}})$ & $N_{{\rm Picard}}(N_{{\rm MINRES}})$\\
		\hline
		$\mathcal{T}_{1}$ & 36(19) & 36(24)\\
		\hline
		$\mathcal{T}_{2}$ & 33(19) & 33(24)\\
		\hline
		$\mathcal{T}_{3}$ & 30(18) & 30(24)\\
		\hline
		$\mathcal{T}_{4}$ & 26(18) & 26(23)\\
		\hline
	\end{tabular}
\end{table}

\begin{figure}
	\begin{centering}
		\includegraphics[scale=0.4]{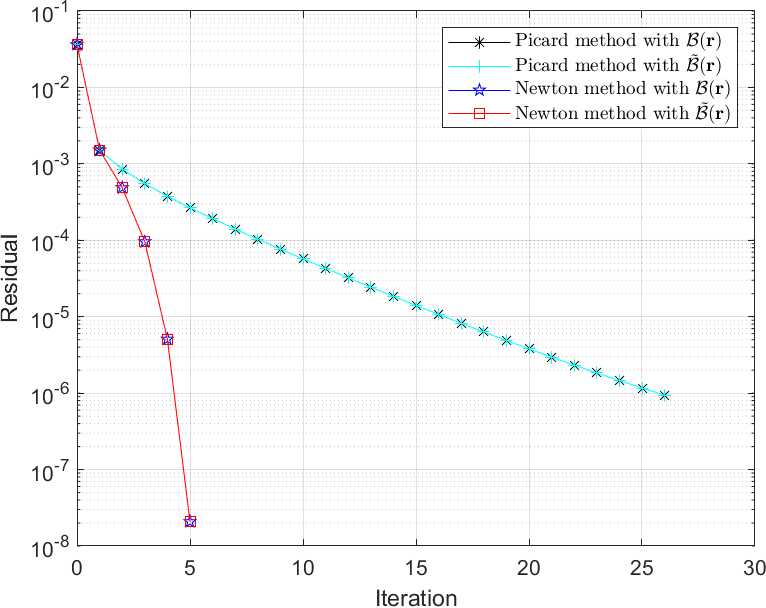}
		\par\end{centering}
	\caption{Convergence histories of Picard method and Newton method (Example \ref{exa:smooth}).\label{fig:ResSmooth}}
\end{figure}

In Table \ref{tab:iterstep2D}, we further give the iteration numbers
of MINRES at each Newton step for the exact and inexact preconditioners.
From the results, we see that the iteration numbers of MINRES is almost invariant with different mesh sizes and iteration numbers. This verifies that our preconditioners
are robust with respect to the iterative steps and mesh sizes. Besides, we again observe that the use of the inexact preconditioners has nearly no impact on the needed number of iterations for MINRES. Using the grid $\mathcal{T}_{4}$, we plot the
convergence histories of MINRES method at each Newton step in Figure \ref{fig:history2D}. It can be seen that the relative residual decrease
rapidly as we expected, which indicates our preconditioners are effective.

\begin{table}[H]
\centering{}\caption{Number of MINRES iterations with the block preconditioners $\mathcal{B}_{h}(\Br)$
and $\widetilde{\mathcal{B}}_{h}(\Br)$ at each Newton step (Example
\ref{exa:smooth}).\label{tab:iterstep2D}}
\begin{tabular}{|c|c|c|c|c|c|c|c|c|c|c|c|}
\hline
\multirow{2}{*}{$h$} & \multicolumn{5}{c|}{$\mathcal{B}_{h}(\Br)$} & \multirow{6}{*}{} & \multicolumn{5}{c|}{$\widetilde{\mathcal{B}}_{h}(\Br)$}\\
\cline{2-6} \cline{3-6} \cline{4-6} \cline{5-6} \cline{6-6} \cline{8-12} \cline{9-12} \cline{10-12} \cline{11-12} \cline{12-12}
 & 1 & 2 & 3 & 4 & 5 &  & 1 & 2 & 3 & 4 & 5\\
\cline{1-6} \cline{2-6} \cline{3-6} \cline{4-6} \cline{5-6} \cline{6-6} \cline{8-12} \cline{9-12} \cline{10-12} \cline{11-12} \cline{12-12}
$\mathcal{T}_{1}$ & 15 & 18 & 22 & 24 & 25 &  & 17 & 23 & 26 & 31 & 33\\
\cline{1-6} \cline{2-6} \cline{3-6} \cline{4-6} \cline{5-6} \cline{6-6} \cline{8-12} \cline{9-12} \cline{10-12} \cline{11-12} \cline{12-12}
$\mathcal{T}_{2}$ & 13 & 18 & 22 & 24 & 25 &  & 17 & 22 & 26 & 31 & 31\\
\cline{1-6} \cline{2-6} \cline{3-6} \cline{4-6} \cline{5-6} \cline{6-6} \cline{8-12} \cline{9-12} \cline{10-12} \cline{11-12} \cline{12-12}
$\mathcal{T}_{3}$ & 13 & 17 & 20 & 23 & 25 &  & 18 & 22 & 26 & 30 & 31\\
\cline{1-6} \cline{2-6} \cline{3-6} \cline{4-6} \cline{5-6} \cline{6-6} \cline{8-12} \cline{9-12} \cline{10-12} \cline{11-12} \cline{12-12}
$\mathcal{T}_{4}$ & 13 & 17 & 19 & 22 & 25 &  & 17 & 22 & 26 & 29 & 31\\
\hline
\end{tabular}
\end{table}

\begin{figure}
\begin{centering}
\begin{tabular}{cc}
\includegraphics[scale=0.35]{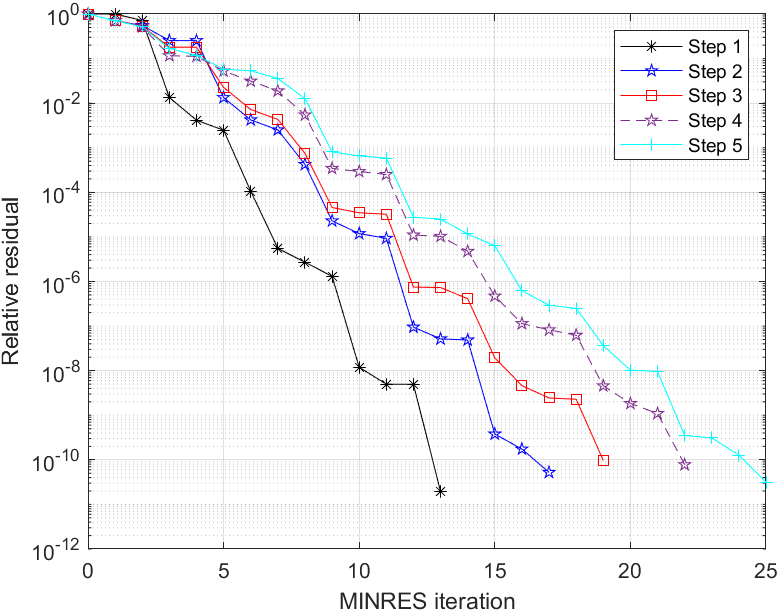} & \includegraphics[scale=0.35]{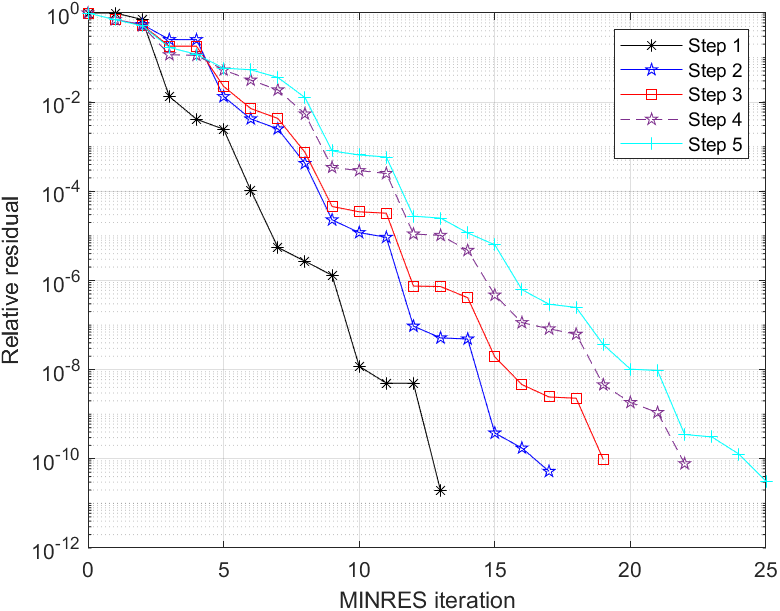}\\
\end{tabular}
\par\end{centering}
\caption{Convergence histories of the preconditioned MINRES method at each Newton
step with $\mathcal{B}_{h}(\Br)$(left) and $\widetilde{\mathcal{B}}_{h}(\Br)$(right).\label{fig:history2D}}
\end{figure}

Finally, we investigate the robustness of the block preconditioners
with respect to the parameters $\alpha$ and $\beta$. We fix the
mesh $\mathcal{T}_{3}$ and vary the parameters. The results for Newton method with the
exact and inexact preconditioners are shown in Table \ref{tab:VaryExact2D}-\ref{tab:VaryInexact2D}.
We can see that the proposed preconditioner
are very robust with respect to the parameters, and that the use of the inexact preconditioner has nearly no impact on the required iterations fro MINRES.

\begin{table}[H]
\begin{centering}
\caption{Iteration counts for Newton method with the block preconditioner $\mathcal{B}_{h}(\Br)$ (Example \ref{exa:smooth}).\label{tab:VaryExact2D}}
\begin{tabular}{|c|c|c|c|c|c|c|}
\hline
 & \multicolumn{6}{c|}{$\alpha$}\\
\hline
\multirow{4}{*}{$\beta$} & $N_{{\rm Newton}}(N_{{\rm MINRES}})$ & 1e5 & 1e3 & 1 & 1e-3 & 1e-5\\
\cline{2-7} \cline{3-7} \cline{4-7} \cline{5-7} \cline{6-7} \cline{7-7}
 & 1 & 7(8) & 7(10) & 5(20) & 2(41) & 1(21)\\
\cline{2-7} \cline{3-7} \cline{4-7} \cline{5-7} \cline{6-7} \cline{7-7}
 & 1e-3 & 14(10) & 13(13) & 8(31) & 2(28) & 2(28)\\
\cline{2-7} \cline{3-7} \cline{4-7} \cline{5-7} \cline{6-7} \cline{7-7}
 & 1e-5 & 13(9) & 13(10) & 10(30) & 2(22) & 3(29)\\
\hline
\end{tabular}
\par\end{centering}
\medskip{}

\centering{}\caption{Iteration counts for Newton method with the block preconditioner $\widetilde{\mathcal{B}}_{h}(\Br)$ (Example \ref{exa:smooth}).\label{tab:VaryInexact2D}}
\begin{tabular}{|c|c|c|c|c|c|c|}
\hline
 & \multicolumn{6}{c|}{$\alpha$}\\
\hline
\multirow{4}{*}{$\beta$} & $N_{{\rm Newton}}(N_{{\rm MINRES}})$ & 1e5 & 1e3 & 1 & 1e-3 & 1e-5\\
\cline{2-7} \cline{3-7} \cline{4-7} \cline{5-7} \cline{6-7} \cline{7-7}
 & 1 & 7(12) & 7(13) & 5(25) & 2(41) & 1(21)\\
\cline{2-7} \cline{3-7} \cline{4-7} \cline{5-7} \cline{6-7} \cline{7-7}
 & 1e-3 & 14(36) & 13(33) & 8(34) & 2(29) & 2(28)\\
\cline{2-7} \cline{3-7} \cline{4-7} \cline{5-7} \cline{6-7} \cline{7-7}
 & 1e-5 & 13(23) & 13(24) & 10(34) & 2(24) & 3(30)\\
\hline
\end{tabular}
\end{table}

\begin{example}
\label{exa:nonsmooth} In this example, we consider a non-smooth problem.
Let $\Omega=(0,1)^{2}$ and denote its center by $x_{\Omega}\coloneqq(0.5,0.5)$.
Write $B_{r}(x_{\Omega})=\left\{ x\in\mathbb{R}^{2}:||x-x_{\Omega}||_{l_{2}}<r\right\} $
with $r=1/3$, the function $f$ is chosen as a characteristic function
$f\coloneqq\chi_{B_{r}(x_{\Omega})}$. Let $I_{h}$ be the standard
nodal interpolation operator to $U_{h}$, the initial guess $u_{h}^{0}$
is taken as $u_{h}^{0}\coloneqq I_{h}f$, and $\left(p_{h}^{0},\lambdabf_{h}^{0}\right)$
is computed by the equations \eqref{eq:weakhfull}.
\end{example}

First, we aim to investigate the convergence rates. From \cite{Strong2003},
if $\alpha=0.02$, the exact solution to the model \eqref{eq:rofmodelu}
is given by
\[
u=\begin{cases}
1-\frac{2\alpha}{r}=0.94 & x\in B_{r}(x_{\Omega}),\\
\frac{2\pi r\alpha}{1-\pi r^{2}}\approx0.03 & x\in\Omega\backslash B_{r}(x_{\Omega}).
\end{cases}
\]
We perform numerical tests for the primal-dual finite element discretization
to \eqref{eq:ROFmodel} with $\beta=1{\rm e}-5$. The errors and convergence
rates for $u$ are displayed in Table \ref{tab:NonsmoothOrder2D}.
Note that the exact solution $u\in L^{\infty}\left(\Omega\right)$,
thus the convergence rates are not perfect as the ones in Example
\ref{exa:smooth}. Even so, the numerical results are in accord with the theoretical
results in Proposition 10.9 of \cite{Bartels2015a} for the standard
finite element discretization to \eqref{eq:rofpmodelu}. The error
estimate of primal-dual finite element discretization is left for
further work, we also refer to \cite{Chambolle2011,Chambolle2020,Chambolle2021,Lai2012,Wang2011} for some discussions on this direction.

\begin{table}[H]
\begin{centering}
\caption{Errors and convergence rates for $u_{h}$ (Example \ref{exa:nonsmooth}).\label{tab:NonsmoothOrder2D}}
\par\end{centering}
\centering{}%
\begin{tabular}{|c|c|c|c|c|}
\hline
$h$ & $||u-u_{h}||_{0}$ & Order & $||u-I_{h}u||_{0}$ & Order\\
\hline
6.25e-02 & 1.12395e-01 & \textemdash{} & 1.17827e-01 & \textemdash{}\\
\hline
3.13e-02 & 7.94646e-02 & 0.50 & 8.35176e-02 & 0.50\\
\hline
1.56e-02 & 6.10573e-02 & 0.38 & 6.33701e-02 & 0.40\\
\hline
7.81e-03 & 4.48697e-02 & 0.44 & 4.17839e-02 & 0.60\\
\hline
\end{tabular}
\end{table}

In Table \ref{tab:Nonsmoothmeshrobust2DNewton}, we present the iteration
numbers of Newton method on various meshes. As predicted from the analysis, the numbers
of MINRES iterations are stable when we vary the mesh size $h$.
Besides, the iteration numbers for $\widetilde{\mathcal{B}}_{h}(\Br)$ are only slightly larger than the numbers for $\mathcal{B}_{h}(\Br)$. This
is expected and the difference is by no means significant. Overall,
we can conclude that our preconditioners are effective and robust
with respect to the mesh size $h$. Figure \ref{fig:MINRES_Step} plots
the iteration numbers of MINRES at each Newton step. Again, we observe relative
robustness with respect to the iterative step. Specifically, Figure
\ref{fig:Damped} shows the damping parameter with respect to the
iterative step on mesh $\mathcal{T}_{4}$. We find that the damping
strategy is crucial to guarantee the global convergence of Newton's
method.

\begin{table}[H]
\centering{}\caption{Iteration counts for Newton method with the block preconditioners $\mathcal{B}_{h}(\Br)$
and $\widetilde{\mathcal{B}}_{h}(\Br)$ (Example \ref{exa:nonsmooth}).\label{tab:Nonsmoothmeshrobust2DNewton}}
\begin{tabular}{|c|c|c|}
\hline
\multirow{2}{*}{mesh} & \multicolumn{1}{c|}{$\mathcal{B}_{h}(\Br)$} & \multicolumn{1}{c|}{$\widetilde{\mathcal{B}}_{h}(\Br)$}\\
\cline{2-3} \cline{3-3}
 & $N_{{\rm Newton}}(N_{{\rm MINRES}})$ & $N_{{\rm Newton}}(N_{{\rm MINRES}})$\\
\hline
$\mathcal{T}_{1}$ & 21(35) & 21(35)\\
\hline
$\mathcal{T}_{2}$ & 23(32) & 23(32)\\
\hline
$\mathcal{T}_{3}$ & 40(31) & 34(31)\\
\hline
$\mathcal{T}_{4}$ & 21(29) & 24(31)\\
\hline
\end{tabular}
\end{table}

\begin{figure}
\begin{centering}
\begin{tabular}{cc}
\includegraphics[scale=0.4]{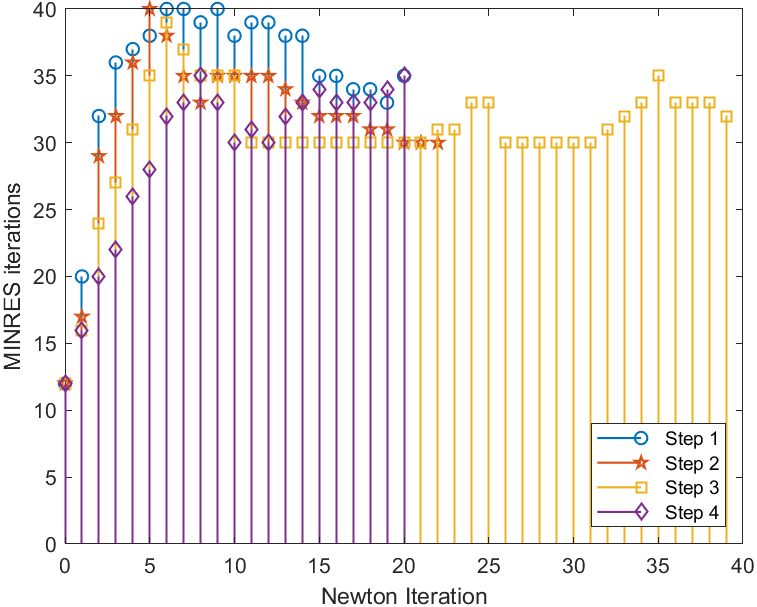} & \includegraphics[scale=0.4]{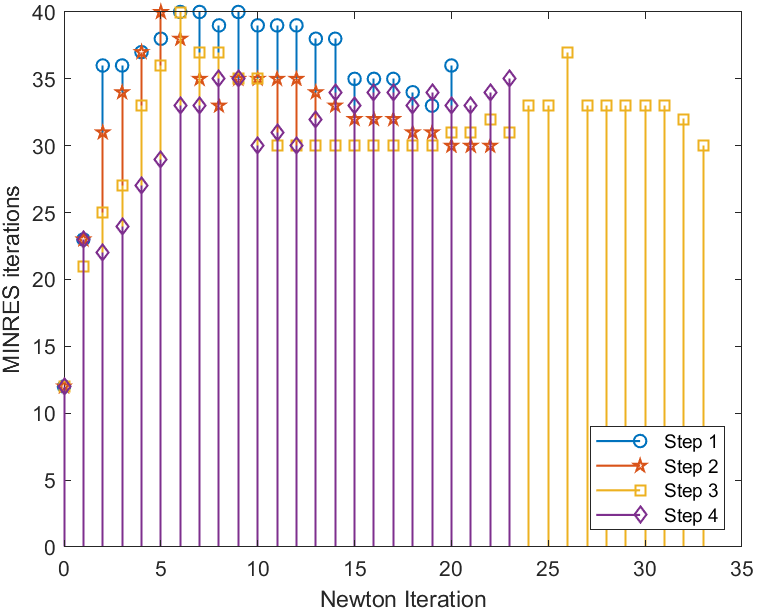}\\
\end{tabular}
\par\end{centering}
\caption{Number of MINRES iterations with the block preconditioners $\mathcal{B}_{h}(\Br)$(left)
and $\widetilde{\mathcal{B}}_{h}(\Br)$(right) at each Newton step
(Example \ref{exa:smooth}).\label{fig:MINRES_Step}}
\end{figure}

\begin{figure}
\begin{centering}
\includegraphics[scale=0.5]{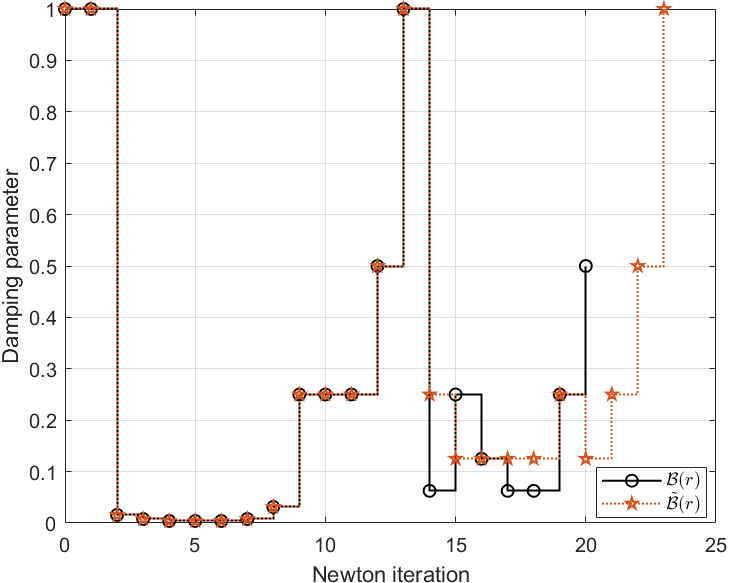}
\par\end{centering}
\caption{Damping parameter at each Newton iteration with $\mathcal{B}_{h}(\Br)$
and $\widetilde{\mathcal{B}}_{h}(\Br)$ on mesh $\mathcal{T}_{4}$.\label{fig:Damped}}
\end{figure}

As in the previous example, we vary the parameters $\alpha$ and $\beta$
to study the robustness of the preconditioners. Table \ref{tab:VaryExact2DNon}-\ref{tab:VaryInexact2DNon}
show the results for Newton method with the exact and inexact preconditioners on mesh
$\mathcal{T}_{3}$. Again, we see that the preconditioners show relative
robustness with respect to the parameters. The inexact preconditioner
requires a slightly higher number of iterations to converge compared
to the exact one, as we saw in the previous example.

\begin{table}
\begin{centering}
\caption{Iteration counts for Newton method with the block preconditioner $\mathcal{B}_{h}(\Br)$. (Example \ref{exa:nonsmooth})\label{tab:VaryExact2DNon}}
\begin{tabular}{|c|c|c|c|c|c|}
	\hline
	& \multicolumn{5}{c|}{$\alpha$}\\
	\hline
	\multirow{5}{*}{$\beta$} & $N_{{\rm Newton}}(N_{{\rm MINRES}})$ & 1e-1 & 5e-2 & 1e-2 & 5e-3\\
	\cline{2-6} \cline{3-6} \cline{4-6} \cline{5-6} \cline{6-6}
	& 1 & 22(28) & 8(35) & 5(46) & 4(49)\\
	\cline{2-6} \cline{3-6} \cline{4-6} \cline{5-6} \cline{6-6}
	& 1e-1 & 31(26) & 6(31) & 10(43) & 9(45)\\
	\cline{2-6} \cline{3-6} \cline{4-6} \cline{5-6} \cline{6-6}
	& 1e-2 & 27(21) & 7(28) & 9(40) & 12(41)\\
	\cline{2-6} \cline{3-6} \cline{4-6} \cline{5-6} \cline{6-6}
	& 1e-3 & 10(18) & 8(25) & 15(39) & 18(37)\\
	\hline
\end{tabular}
\par\end{centering}
\medskip{}

\centering{}\caption{Iteration counts for Newton method with the block preconditioner $\widetilde{\mathcal{B}}_{h}(\Br)$ (Example \ref{exa:nonsmooth}).\label{tab:VaryInexact2DNon}}
\begin{tabular}{|c|c|c|c|c|c|}
	\hline
	& \multicolumn{5}{c|}{$\alpha$}\\
	\hline
	\multirow{5}{*}{$\beta$} & $N_{{\rm Newton}}(N_{{\rm MINRES}})$ & 1e-1 & 5e-2 & 1e-2 & 5e-3\\
	\cline{2-6} \cline{3-6} \cline{4-6} \cline{5-6} \cline{6-6}
	& 1 & 22(33) & 8(42) & 5(46) & 4(49)\\
	\cline{2-6} \cline{3-6} \cline{4-6} \cline{5-6} \cline{6-6}
	& 1e-1 & 31(30) & 6(35) & 10(44) & 9(45)\\
	\cline{2-6} \cline{3-6} \cline{4-6} \cline{5-6} \cline{6-6}
	& 1e-2 & 37(25) & 7(30) & 9(42) & 12(41)\\
	\cline{2-6} \cline{3-6} \cline{4-6} \cline{5-6} \cline{6-6}
	& 1e-3 & 10(22) & 8(27) & 18(39) & 18(37)\\
	\hline
\end{tabular}
\end{table}

For comparison, we present the number of iterations
required by Picard method for different $\alpha$ and $\beta$ on
mesh $\mathcal{T}_{3}$ in Tables
\ref{tab:VaryExact2DNonPicard}-\ref{tab:VaryInexact2DNonPicard}. Compared with the results of Newton method,
we can see that the Newton iteration converges rapidly than
the Picard iteration for the considered parameters.
For the case of  $\alpha =1\rm{e}-2,\beta=1\rm{e}-3$, Figure \ref{fig:Res} plots the convergence histories of Picard method and Newton method with $\mathcal{B}_{h}(\Br)$ and $\widetilde{\mathcal{B}}_{h}(\Br)$. As we conclude, the Newton iteration behaves similarly to the Picard iteration in the early stages but converges rapidly in the end stages.

\begin{table}
\begin{centering}
	\caption{Iteration counts for Picard method with the block preconditioner $\mathcal{B}_{h}(\Br)$.
		(Example \ref{exa:nonsmooth})\label{tab:VaryExact2DNonPicard}}
	\par\end{centering}
\begin{centering}
	\begin{tabular}{|c|c|c|c|c|c|}
		\hline
		& \multicolumn{5}{c|}{$\alpha$}\tabularnewline
		\hline
		\multirow{5}{*}{$\beta$} & $N_{{\rm Picard}}(N_{{\rm MINRES}})$ & 1e-1 & 5e-2 & 1e-2 & 5e-3\\
		\cline{2-6} \cline{3-6} \cline{4-6} \cline{5-6} \cline{6-6}
		& 1 & 20(27) & 26(31) & 13(44) & 9(50)\tabularnewline
		\cline{2-6} \cline{3-6} \cline{4-6} \cline{5-6} \cline{6-6}
		& 1e-1 & 39(22) & 37(27) & 18(37) & 12(42)\tabularnewline
		\cline{2-6} \cline{3-6} \cline{4-6} \cline{5-6} \cline{6-6}
		& 1e-2 & 68(19) & 52(24) & 23(31) & 16(35)\tabularnewline
		\cline{2-6} \cline{3-6} \cline{4-6} \cline{5-6} \cline{6-6}
		& 1e-3 & 90(18) & 71(22) & 29(29) & 20(32)\tabularnewline
		\hline
	\end{tabular}
	\par\end{centering}
\medskip{}

\begin{centering}
	\caption{Iteration counts for Picard method with the block preconditioner $\widetilde{\mathcal{B}}_{h}(\Br)$
		(Example \ref{exa:nonsmooth}).\label{tab:VaryInexact2DNonPicard}}
	\par\end{centering}
\begin{centering}
	\begin{tabular}{|c|c|c|c|c|c|}
		\hline
		& \multicolumn{5}{c|}{$\alpha$}\tabularnewline
		\hline
		\multirow{5}{*}{$\beta$} & $N_{{\rm Picard}}(N_{{\rm MINRES}})$ & 1e-1 & 5e-2 & 1e-2 & 5e-3\\
		\cline{2-6} \cline{3-6} \cline{4-6} \cline{5-6} \cline{6-6}
		& 1 & 20(33) & 26(39) & 13(45) & 9(50)\tabularnewline
		\cline{2-6} \cline{3-6} \cline{4-6} \cline{5-6} \cline{6-6}
		& 1e-1 & 39(28) & 37(34) & 18(42) & 12(45)\tabularnewline
		\cline{2-6} \cline{3-6} \cline{4-6} \cline{5-6} \cline{6-6}
		& 1e-2 & 68(24) & 52(31) & 23(38) & 16(42)\tabularnewline
		\cline{2-6} \cline{3-6} \cline{4-6} \cline{5-6} \cline{6-6}
		& 1e-3 & 90(22) & 71(29) & 29(36) & 20(39)\tabularnewline
		\hline
	\end{tabular}
	\par\end{centering}
\centering{}%

\end{table}

\begin{figure}
	\begin{centering}
		\includegraphics[scale=0.4]{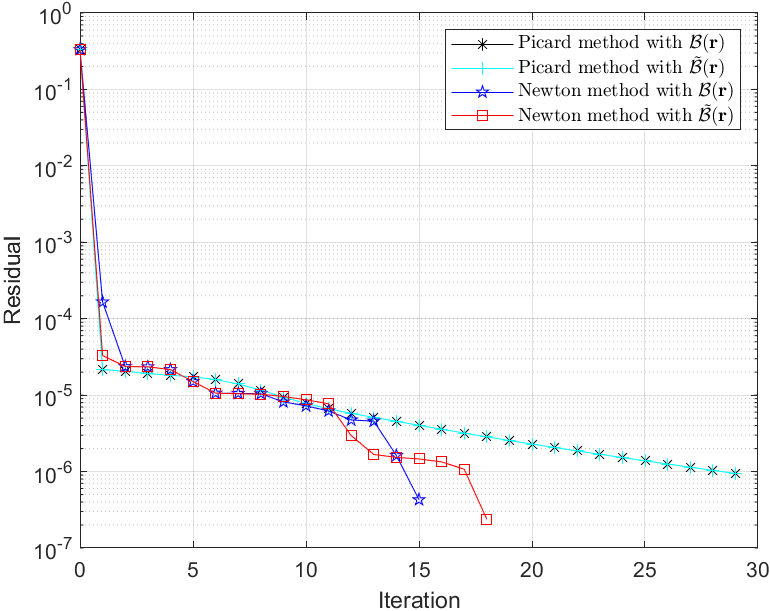}
		\par\end{centering}
	\caption{Convergence histories of Picard method and Newton method (Example \ref{exa:nonsmooth}).\label{fig:Res}}
\end{figure}

\begin{example}
	\label{exa:bench}This example we consider is a benchmark problem
	as seen in \cite{Bartels2012,Bartels2018,Tian2018}. Let $\Omega=(0,1)^{2}$
	and denote its center by $x_{\Omega}$. Given a triangulation $\mathcal{T}_{h}$
	of $\Omega$, we define a randomly perturbed $\xi_{h}\in U_{h}$,
	whose coefficient vector is sampled from the normally distributed.
	Write $B_{r}^{p}(x_{\Omega})=\left\{ x\in\mathbb{R}^{d}:|x-x_{\Omega}|_{{l}_p}<r\right\} $
	with $r=1/3$ and $p\in[1,\infty]$, the function $f$ is set by a
	characteristic function $f_{0}\coloneqq\chi_{B_{r}^{p}(x_{\Omega})}$
	that is mesh-dependent perturbed $\xi_{h}$, precisely,
	\[
	f=f_{0}+\delta\xi_{h},\quad\text{with}\quad\delta=0.1.
	\]
	The parameters are given by $\alpha=5{\rm e-}2$ and $\beta=1{\rm e-3}.$
	The initial value for Picard method is set as Example \ref{exa:nonsmooth}. While for  Newton method, the initial value is taken as 5-step  iterations of Picard method to improve its efficiency.
\end{example}

To experimentally study the effectiveness of the proposed method,
we run Picard method and Newton with $p=1,2,\infty$ on mesh ${\cal T}_{4}$.
Figure \ref{fig:noisy} displays the initial data and the outputs
of the iterative schemes. Notice that Newton method and Picard method
yield very similar results. From the outputs, we see that noise is removed effectively. The boundary is slightly smoothed, especially the corner and the numerical results indeed show the inherited properties of the ROF model.

\begin{figure}
		\begin{centering}
			\begin{tabular}{ccc}
				$p=1$ & $p=2$ & $p=\infty$\tabularnewline
				\includegraphics[scale=0.25]{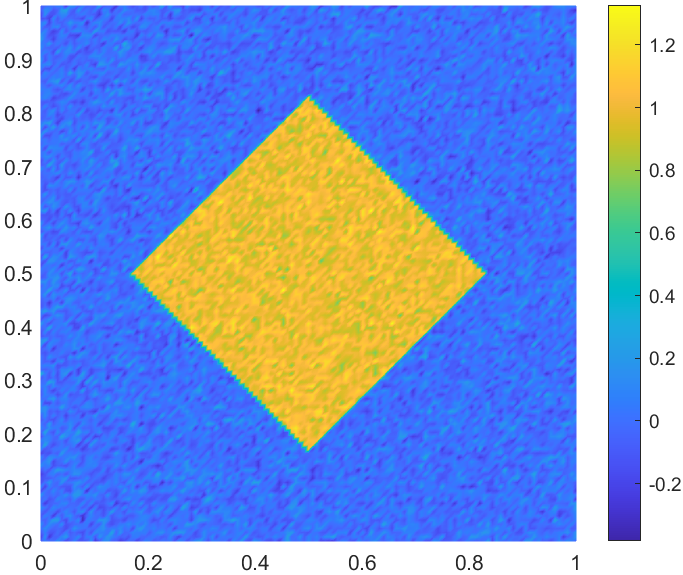} & \includegraphics[scale=0.25]{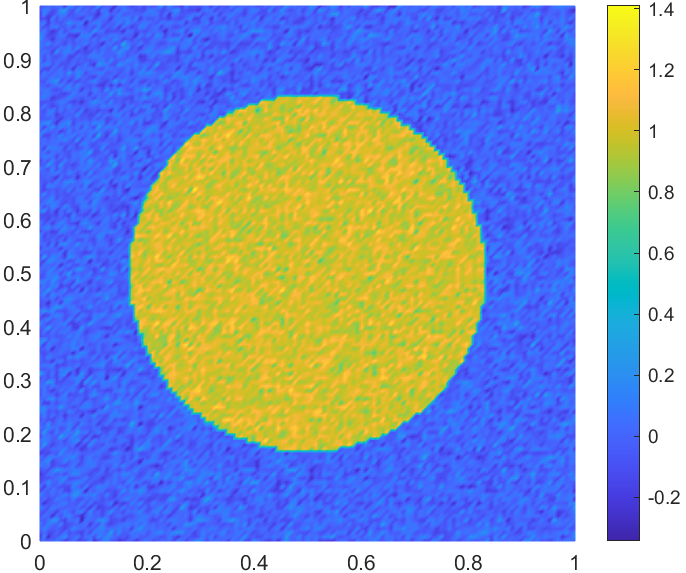} & \includegraphics[scale=0.25]{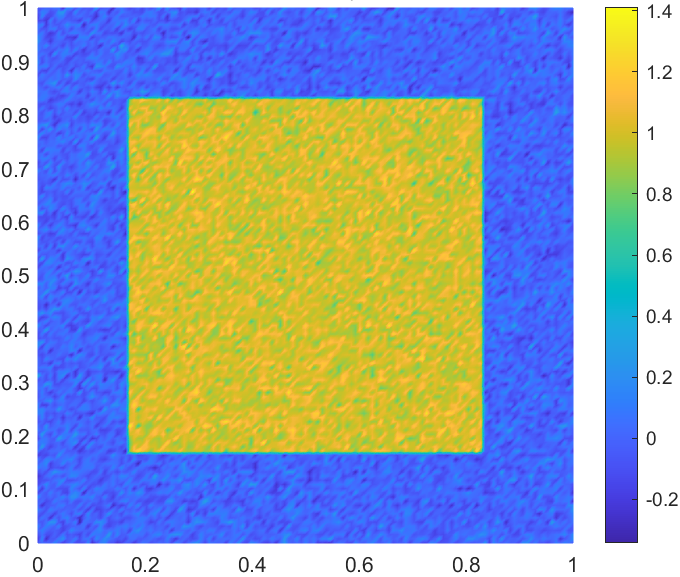}\tabularnewline
				\includegraphics[scale=0.25]{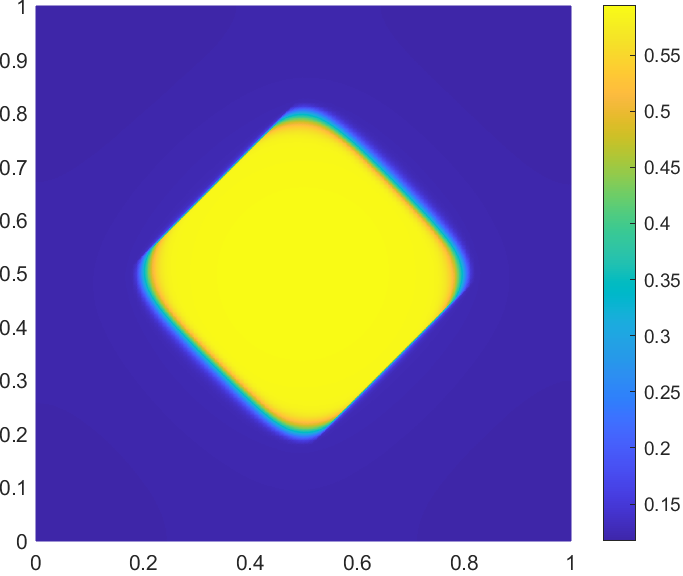} & \includegraphics[scale=0.25]{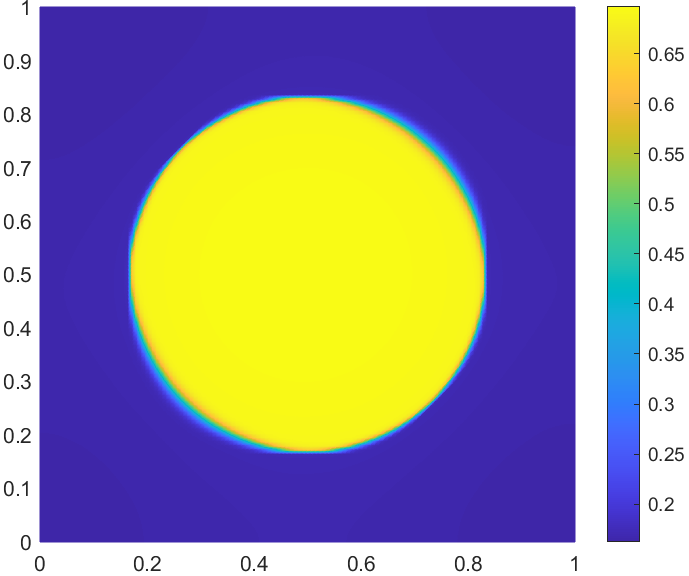} & \includegraphics[scale=0.25]{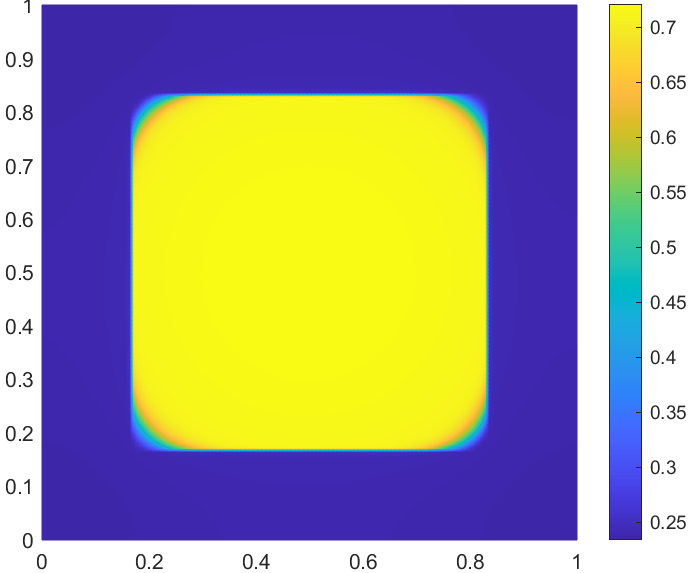}\tabularnewline
				\includegraphics[scale=0.25]{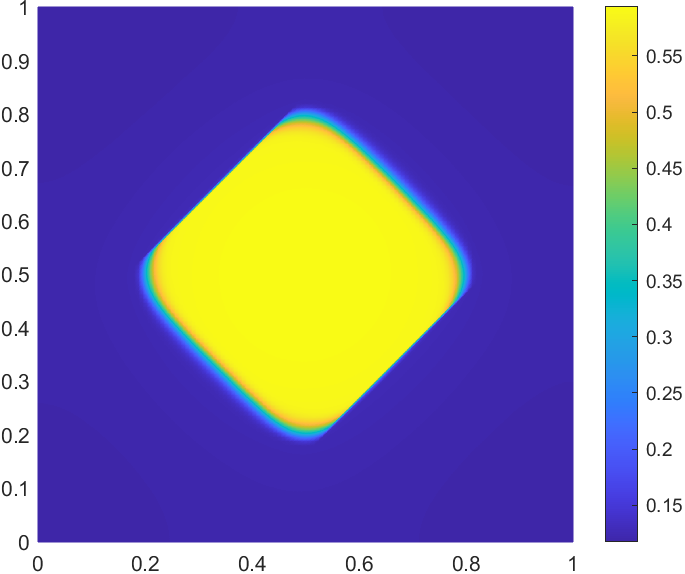} & \includegraphics[scale=0.25]{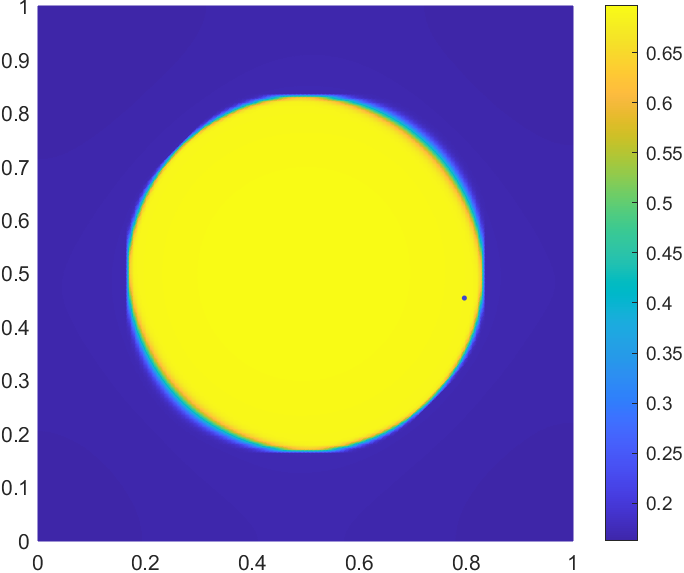} & \includegraphics[scale=0.25]{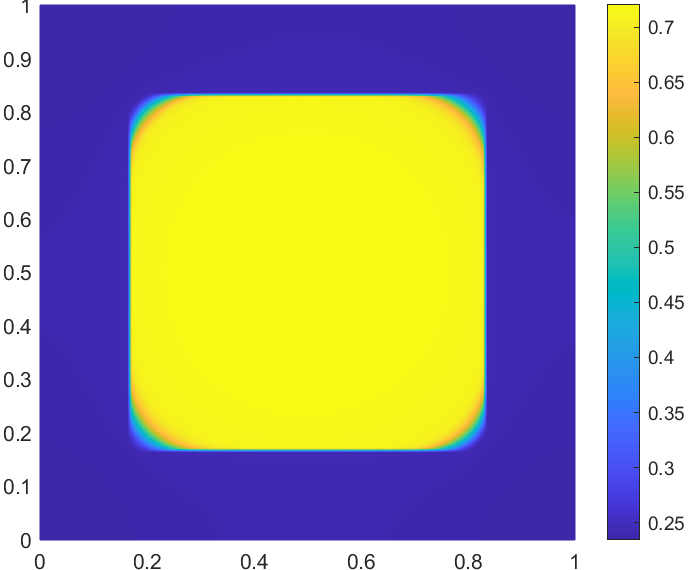}\tabularnewline
			\end{tabular}
			\par\end{centering}
		\caption{Noisy image, denoised image with Picard method and denoised image
			with Newton method (from top to bottom).
			\label{fig:noisy}}
\end{figure}

In Table \ref{tab:NonsmoothNoisy}, we display the iteration numbers
for Picard method and Newton method with $\widetilde{{\cal B}_{h}}(\Br)$.
We again observe that Newton method needs less iteration numbers than
Picard method. We further plot the convergence histories of this experiment
in Figure \ref{fig:Newton}. We see that  Newton method converges much faster than Picard method after a few damped Newton updatings.

\begin{table}
\centering{}\caption{Iteration counts for Picard method and Newton method (Example \ref{exa:bench}).\label{tab:NonsmoothNoisy}}
\begin{tabular}{|c|c|c|}
\hline
\multirow{1}{*}{$p$} & $N_{{\rm Picard}}(N_{{\rm MINRES}})$ & $N_{{\rm Newton}}(N_{{\rm MINRES}})$\tabularnewline
\hline
1 & 66(29) & 16(41)\tabularnewline
\hline
2 & 67(30) & 19(39)\tabularnewline
\hline
$\infty$ & 86(29) & 18(43)\tabularnewline
\hline
\end{tabular}
\end{table}

\begin{figure}
\begin{centering}
		\begin{tabular}{ccc}
			\includegraphics[scale=0.25]{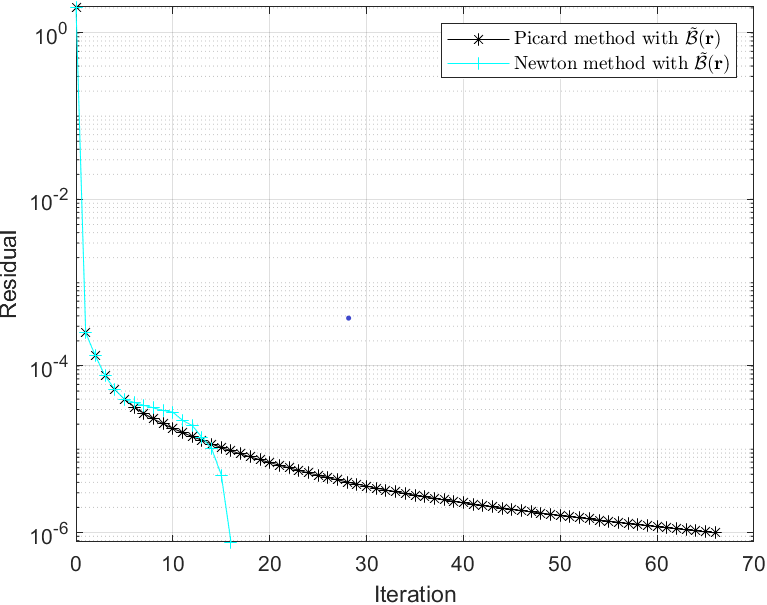} & \includegraphics[scale=0.25]{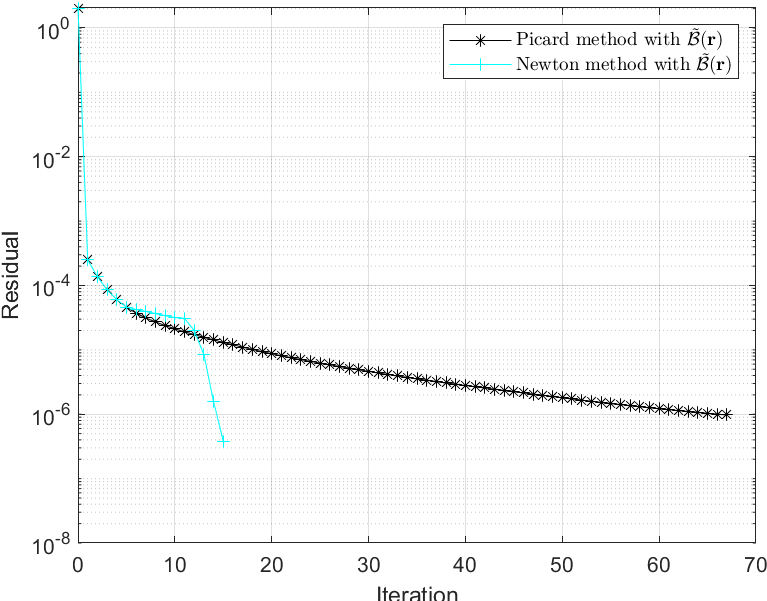} & \includegraphics[scale=0.25]{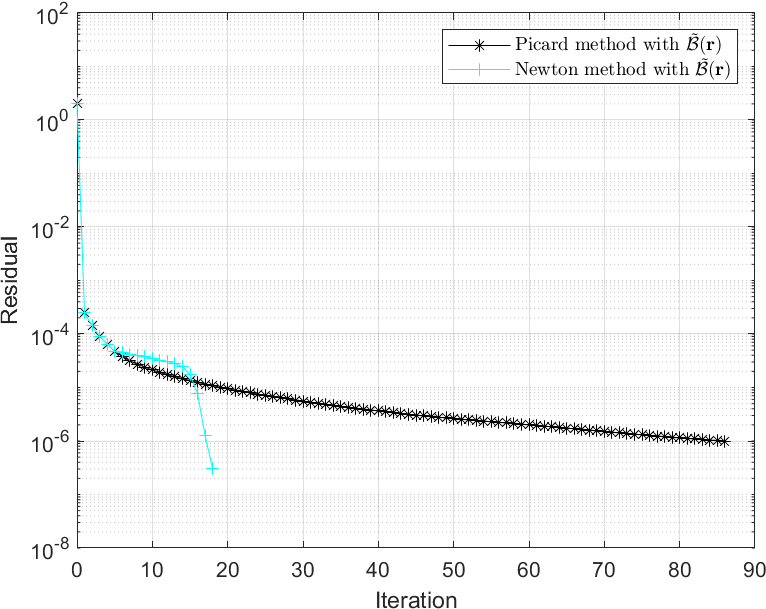}\tabularnewline
		\end{tabular}
\par\end{centering}
	\caption{Convergence histories of Picard method and Newton method for $p=1,2,\infty$
		(from left to right) (Example \ref{exa:bench}).\label{fig:Newton}}
\end{figure}

\section{Conclusions\label{sec:conlu}}

In this paper, we propose a preconditioned Newton solver for primal-dual
finite element approximation of total variation minimization and minimum surface problems. We develop some block diagonal  preconditioners
for the discrete problems at each Newton iteration, which are robust
with respect to the mesh size, the penalization parameter, the regularization
parameter, and the iterative step. We further prove that the resulting
preconditioned MINRES converges uniformly. The theoretical findings
are demonstrated by numerical experiments.

%\section*{References}
\bibliographystyle{siamplain}
\bibliography{ref}

\begin{thebibliography}{10}

\bibitem{Acar1994}
{\sc R.~Acar and C.~R. Vogel}, {\em Analysis of bounded variation penalty
  methods for ill-posed problems}, Inverse Problems, 10 (1994), pp.~1217--1229.

\bibitem{Bartels2012}
{\sc S.~Bartels}, {\em Total variation minimization with finite elements:
  convergence and iterative solution}, SIAM J. Numer. Anal., 50 (2012),
  pp.~1162--1180.

\bibitem{Bartels2015a}
{\sc S.~Bartels}, {\em Numerical methods for nonlinear partial differential
  equations}, vol.~47 of Springer Series in Computational Mathematics,
  Springer, Cham, 2015.

\bibitem{Bartels2018}
{\sc S.~Bartels, L.~Diening, and R.~H. Nochetto}, {\em Unconditional stability
  of semi-implicit discretizations of singular flows}, SIAM J. Numer. Anal., 56
  (2018), pp.~1896--1914.

\bibitem{Bartels2015}
{\sc S.~Bartels, R.~H. Nochetto, and A.~J. Salgado}, {\em A total variation
  diminishing interpolation operator and applications}, Math. Comp., 84 (2015),
  pp.~2569--2587.

\bibitem{Briggs2000}
{\sc W.~L. Briggs, V.~E. Henson, and S.~F. McCormick}, {\em A multigrid
  tutorial}, Society for Industrial and Applied Mathematics (SIAM),
  Philadelphia, PA, second~ed., 2000.

\bibitem{Casas1999}
{\sc E.~Casas, K.~Kunisch, and C.~Pola}, {\em Regularization by functions of
  bounded variation and applications to image enhancement}, Appl. Math. Optim.,
  40 (1999), pp.~229--257.

\bibitem{Chambolle2004}
{\sc A.~Chambolle}, {\em An algorithm for total variation minimization and
  applications}, J. Math. Imaging Vision, 20 (2004), pp.~89--97.
\newblock Special issue on mathematics and image analysis.

\bibitem{Chambolle2011}
{\sc A.~Chambolle, S.~E. Levine, and B.~J. Lucier}, {\em An upwind
  finite-difference method for total variation-based image smoothing}, SIAM J.
  Imaging Sci., 4 (2011), pp.~277--299.

\bibitem{Chambolle2020}
{\sc A.~Chambolle and T.~Pock}, {\em Crouzeix-{R}aviart approximation of the
  total variation on simplicial meshes}, J. Math. Imaging Vision, 62 (2020),
  pp.~872--899.

\bibitem{Chambolle2021}
{\sc A.~Chambolle and T.~Pock}, {\em Approximating the total variation with
  finite differences or finite elements}, in Geometric partial differential
  equations. {P}art {II}, vol.~22 of Handb. Numer. Anal.,
  Elsevier/North-Holland, Amsterdam, [2021] \copyright 2021, pp.~383--417.

\bibitem{Chan1999}
{\sc T.~F. Chan, G.~H. Golub, and P.~Mulet}, {\em A nonlinear primal-dual
  method for total variation-based image restoration}, SIAM J. Sci. Comput., 20
  (1999), pp.~1964--1977.

\bibitem{Chan2003}
{\sc T.~F. Chan and J.~Shen}, {\em On the role of the {BV} image model in image
  restoration}, in Recent advances in scientific computing and partial
  differential equations ({H}ong {K}ong, 2002), vol.~330 of Contemp. Math.,
  Amer. Math. Soc., Providence, RI, 2003, pp.~25--41.

\bibitem{Chen2008}
{\sc L.~Chen}, {\em {$i$FEM}: an integrated finite element methods package in
  {MATLAB}}, Technical Report, University of California at Irvine,  (2009).

\bibitem{Dobson1997}
{\sc D.~C. Dobson and C.~R. Vogel}, {\em Convergence of an iterative method for
  total variation denoising}, SIAM J. Numer. Anal., 34 (1997), pp.~1779--1791.

\bibitem{Goldstein2009}
{\sc T.~Goldstein and S.~Osher}, {\em The split {B}regman method for
  {$L1$}-regularized problems}, SIAM J. Imaging Sci., 2 (2009), pp.~323--343.

\bibitem{Hackbusch2016}
{\sc W.~Hackbusch}, {\em Iterative solution of large sparse systems of
  equations}, vol.~95 of Applied Mathematical Sciences, Springer, [Cham],
  second~ed., 2016.

\bibitem{Winther2009}
{\sc Q.~Hu, X.~Tai, and R.~Winther}, {\em {A saddle point approach to the
  computation of harmonic maps}}, Siam J. Numer. Anal, 47 (2009),
  pp.~1500--1523.

\bibitem{Lai2012}
{\sc M.-J. Lai and L.~Matamba~Messi}, {\em Piecewise linear approximation of
  the continuous {R}udin-{O}sher-{F}atemi model for image denoising}, SIAM J.
  Numer. Anal., 50 (2012), pp.~2446--2466.

\bibitem{Lao2021}
{\sc D.~Lao and S.~Zhao}, {\em Fundamental theories and their applications of
  the calculus of variations}, Springer, Singapore; Beijing Institute of
  Technology Press, Beijing, 2021.

\bibitem{Lee2019}
{\sc C.-O. Lee, E.-H. Park, and J.~Park}, {\em A finite element approach for
  the dual {R}udin-{O}sher-{F}atemi model and its nonoverlapping domain
  decomposition methods}, SIAM J. Sci. Comput., 41 (2019), pp.~B205--B228.

\bibitem{Daniel2012}
{\sc D.~Liberzon}, {\em Calculus of variations and optimal control theory},
  Princeton University Press, Princeton, NJ, 2012.
\newblock A concise introduction.

\bibitem{Mardal2011}
{\sc K.-A. Mardal and R.~Winther}, {\em Preconditioning discretizations of
  systems of partial differential equations}, Numer. Linear Algebra Appl., 18
  (2011), pp.~1--40.

\bibitem{Marquina2000}
{\sc A.~Marquina and S.~Osher}, {\em Explicit algorithms for a new time
  dependent model based on level set motion for nonlinear deblurring and noise
  removal}, SIAM J. Sci. Comput., 22 (2000), pp.~387--405.

\bibitem{Osher2003}
{\sc S.~Osher and R.~Fedkiw}, {\em Level set methods and dynamic implicit
  surfaces}, vol.~153 of Applied Mathematical Sciences, Springer-Verlag, New
  York, 2003.

\bibitem{Rudin1992}
{\sc L.~I. Rudin, S.~Osher, and E.~Fatemi}, {\em Nonlinear total variation
  based noise removal algorithms}, Phys. D, 60 (1992), pp.~259--268.
\newblock Experimental mathematics: computational issues in nonlinear science
  (Los Alamos, NM, 1991).

\bibitem{Strong2003}
{\sc D.~Strong and T.~Chan}, {\em Edge-preserving and scale-dependent
  properties of total variation regularization}, Inverse Problems, 19 (2003),
  pp.~S165--S187.
\newblock Special section on imaging.

\bibitem{Tai2009}
{\sc X.-C. Tai and C.~Wu}, {\em Augmented lagrangian method, dual methods and
  split bregman iteration for rof model}, in Scale Space and Variational
  Methods in Computer Vision, X.-C. Tai, K.~M{\o}rken, M.~Lysaker, and K.-A.
  Lie, eds., Berlin, Heidelberg, 2009, Springer Berlin Heidelberg,
  pp.~502--513.

\bibitem{Tian2018}
{\sc W.~Tian and X.~Yuan}, {\em Convergence analysis of primal-dual based
  methods for total variation minimization with finite element approximation},
  J. Sci. Comput., 76 (2018), pp.~243--274.

\bibitem{Vogel1996}
{\sc C.~R. Vogel and M.~E. Oman}, {\em Iterative methods for total variation
  denoising}, SIAM J. Sci. Comput., 17 (1996), pp.~227--238.
\newblock Special issue on iterative methods in numerical linear algebra
  (Breckenridge, CO, 1994).

\bibitem{Wang2011}
{\sc J.~Wang and B.~J. Lucier}, {\em Error bounds for finite-difference methods
  for {R}udin-{O}sher-{F}atemi image smoothing}, SIAM J. Numer. Anal., 49
  (2011), pp.~845--868.

\bibitem{Wu2010}
{\sc C.~Wu and X.-C. Tai}, {\em Augmented {L}agrangian method, dual methods,
  and split {B}regman iteration for {ROF}, vectorial {TV}, and high order
  models}, SIAM J. Imaging Sci., 3 (2010), pp.~300--339.

\bibitem{Xu2010}
{\sc J.~Xu, X.-C. Tai, and L.-L. Wang}, {\em A two-level domain decomposition
  method for image restoration}, Inverse Probl. Imaging, 4 (2010),
  pp.~523--545.

\bibitem{Xu2017}
{\sc J.~Xu and L.~Zikatanov}, {\em Algebraic multigrid methods}, Acta Numer.,
  26 (2017), pp.~591--721.

\bibitem{Yao2008}
{\sc C.~H. Yao}, {\em Finite element approximation for {TV} regularization},
  Int. J. Numer. Anal. Model., 5 (2008), pp.~516--526.

\end{thebibliography}

\end{document}